\documentclass[a4paper, 11pt, reqno]{amsart}
\usepackage{graphicx,verbatim}
\usepackage{mathtools}

\usepackage{enumitem}
\usepackage[pdfstartview={XYZ null null null}]{hyperref} % bookmarksopen
\usepackage{hyperref}
\newtheorem{theorem}{Theorem}[section]
\newtheorem{lemma}[theorem]{Lemma}
\newtheorem{proposition}[theorem]{Proposition}
\newtheorem{corollary}[theorem]{Corollary}
\newtheorem{conjecture}[theorem]{Conjecture}
\newtheorem{claim}[theorem]{Claim}

\newtheorem{definition}[theorem]{Definition}

\numberwithin{equation}{section}

\newcommand{\defn}{\emph}

\newcommand{\N}{\mathbb{N}}

\newcommand{\R}{\mathbb{R}}

\newcommand{\weights}[3]{{#1}_{#2}(#3)} % symbol edge clique
\newcommand{\clique}[2]{K^{(#1)}_{#2}}
\newcommand{\cliques}[2]{\mathcal K^{(#1)}_{#2}}
\newcommand{\indicator}{\mathrm 1}

\newcommand{\extensions}[2]{\kappa_{#1}^{(#2)}}

\newcommand{\KK}{\mathcal{K}}

\newcommand{\HH}{\mathcal{H}}
\renewcommand{\AA}{\mathcal A}

\renewcommand{\d}{\delta}
\renewcommand{\a}{\alpha}

\newcommand{\bigdelta}{c} % things that used to be called \delta but are not close to 0

\def\eps{{\varepsilon}}
\renewcommand{\epsilon}{\varepsilon}

\renewcommand{\subset}{\subseteq}

\newcommand{\COMMENT}[1]{}

\title{Fractional Clique Decompositions of Dense Graphs and Hypergraphs}

\author{Ben Barber, % \and
        Daniela K\"uhn, % \and
       Allan Lo, % \and
      Richard Montgomery \and
     Deryk Osthus}
\address{School of Mathematics, University of Birmingham, Birmingham, B15 2TT, UK}
\email{\{b.a.barber, d.kuhn, s.a.lo, r.h.montgomery, d.osthus\}@bham.ac.uk}
\thanks{The research leading to these results was partially supported by the  European Research Council under the European Union's Seventh Framework Programme (FP/2007--2013) / ERC Grant Agreement n. 258345 (B.~Barber,  D.~K\"uhn and R.~Montgomery) and 306349 (D.~Osthus).
The research was also partially supported by the EPSRC, grant no. EP/M009408/1 (D.~K\"uhn and D.~Osthus).
}

\begin{document}
\date{\today}
\begin{abstract}
%Let~$F$ and~$G$ be $k$-uniform hypergraphs and let $\mathcal F(G)$ be the set of copies of~$F$ in~$G$.
%We say that $G$ has a \defn{fractional $F$-decomposition} if there exists a function $\omega : \mathcal{F}(G) \to [0,1]$ such that for every edge $e \in E(G)$, $\sum_{F \in \mathcal{F}(G) \colon  e \in E(F)} \omega (F) = 1$.
Our main result is that every graph $G$ on $n\ge 10^4r^3$ vertices with minimum degree $\delta(G) \ge (1 - 1 / 10^4 r^{3/2} ) n$ has a fractional 
$K_r$-decomposition. 
Combining this result with recent work of Barber, K\"uhn, Lo and Osthus leads 
to the best known minimum degree thresholds for exact (non-fractional) $F$-decompositions for a wide class of graphs~$F$ (including large cliques).
%as well as the best threshold in terms of $|F|$.
For general $k$-uniform hypergraphs, we give a short argument which shows 
that there exists a constant $c_k>0$ such that every $k$-uniform hypergraph $G$ on $n$ vertices with minimum codegree at least 
$(1- c_k /r^{2k-1})  n $ has a fractional $K^{(k)}_r$-decomposition, where $K^{(k)}_r$ is the complete $k$-uniform hypergraph on $r$ vertices.
(Related fractional decomposition results for triangles have been obtained by Dross and for hypergraph cliques by Dukes as well as Yuster.)
All the above new results involve purely combinatorial arguments. 
In particular, this yields a combinatorial proof of Wilson's theorem that 
every large $F$-divisible complete graph has an $F$-decomposition.

%The \defn{fractional $F$-decomposition threshold} $\delta^*_F$ is the least $\delta$ such that, if a hypergraph $G$ on $n$ vertices has minimum degree $(\delta + o(1))n$, then it has a fractional $F$-decomposition.
%We improve the known bounds on $\delta^*_{\clique k r}$ asymptotically in $r$, for each fixed $k$.
%We prove that $\delta^*_{\clique k r} \leq 1- \Omega(1/r^{2k-1})$, improving the bound $\delta^*_{\clique k r} \leq 1- \Omega(1/r^{2k})$ due to Dukes.
%For graphs (i.e.\ $k=2$) we further improve this to show that $\delta^*_{K_r} \leq 1- \Omega(1/r^{3/2})$.
%Combining this result with recent work of Barber, K\"uhn, Lo and Osthus leads to the best known minimum degree thresholds for exact (non-fractional) $F$-decompositions for a wide class of graphs.
\end{abstract}

\maketitle 

\section{Introduction and results}
\subsection{(Fractional) decompositions of graphs}
We say that a $k$-uniform hypergraph $G$ has an \emph{$F$-decomposition} if its edge set $E(G)$ can be partitioned into copies of $F$.
A natural relaxation is that of a fractional decomposition. To define this, let $\mathcal F(G)$ be the set of copies of~$F$ in $G$.
A \defn{fractional $F$-decomposition} is a function $\omega : \mathcal F(G) \to [0,1]$ such that, for each $e \in E(G)$,
\begin{equation} \label{packing}
\sum_{F \in \mathcal F(G) \colon e \in E(F)} \omega(F) = 1.
\end{equation}
Note that every $F$-decomposition is a fractional $F$-decomposition where $\omega(F) \in \{0, 1\}$.
As a partial converse, Haxell and R\"odl~\cite{haxellrodl} used Szemer\'edi's regularity lemma to show that the existence of a fractional $F$-decomposition of a graph $G$ implies the existence of an approximate $F$-decomposition of $G$,
i.e.~a set of edge-disjoint copies of $F$ in $G$ which cover almost all edges of $G$
(their main result is more general than this). R\"odl, Schacht, Siggers and Tokushige~\cite{RSST} later generalised this result to $k$-uniform hypergraphs.

The study of $F$-decompositions of cliques is central to design theory and has a long and rich history. 
In 1847, Kirkman~\cite{Kirkman} showed that $K_n$ has a $K_3$-decomposition if and only if $n \equiv 1,3 \mod 6$. 
More generally, we say that a graph $G$ is \defn{$F$-divisible} if $e(F)$ divides $e(G)$ 
and the greatest common divisor of the degrees of $F$ divides the degree of every vertex of $G$.
If $G$ has an $F$-decomposition then it is certainly $F$-divisible.
Wilson~\cite{wilson1,wilson2,wilson3,Wilson} proved that if $G$ is a large complete graph, then this necessary condition is also sufficient.

For a given graph $F$, it is probably not possible to find a satisfactory characterization of all graphs $G$ which have an $F$-decomposition.
This is supported by the fact that Dor and Tarsi~\cite{npcomplete} proved that determining whether a graph $G$ has an $F$-decomposition is NP-complete if $F$ has a connected component with at least $3$ edges.
However, it is natural to ask whether one can extend Kirkman's result and Wilson's theorem to all dense graphs.
In particular, Nash-Williams made the following conjecture on triangle decompositions.
\begin{conjecture}[Nash-Williams~\cite{NashWilliams}] \label{conjecture}
There exists $N \in \mathbb{N}$ so that for all $n \ge N$, if $G$ is a $K_3$-divisible graph on $n$ vertices and $\delta(G) \ge 3n/4$, then $G$ has a $K_3$-decomposition. 
\end{conjecture}

There has been considerable recent progress towards this conjecture.
The first result towards the conjecture was obtained by Gustavsson~\cite{gustavsson} who showed that, for every fixed graph $F$, there exists $\epsilon = \epsilon(F) >0$ and $n_0= n_0(F)$ such that every $F$-divisible graph $G$ on $n \ge n_0$ vertices with minimum degree $\delta(G) \ge (1- \epsilon) n $ has an $F$-decomposition.
The bound on $\eps (F)$ claimed by Gustavsson is around $10^{-37}|F|^{-94}$. 

Recently, Barber, K\"uhn, Lo and Osthus~\cite{BKLO} significantly improved the bound on $\eps(F)$
by establishing a connection to fractional decompositions.
For a graph $F$ and $n \in \mathbb N$, let $\delta^*_F(n)$ be the infimum over all $\bigdelta$ such that every graph $G$ on $n$ vertices with $\delta(G) \geq \bigdelta n$ has a fractional $F$-decomposition.
We call $\delta^*_F := \limsup_{n \to \infty} \delta^*_F(n)$ the \defn{fractional $F$-decomposition threshold}.
The main results in~\cite{BKLO} imply the following.
\begin{theorem}[Barber, K\"uhn, Lo and Osthus \cite{BKLO}] \label{exact-decompositions}
Let $F$ be a graph, let $\epsilon > 0$ and let $n$ be sufficiently large.
Let $G$ be an $F$-divisible graph on $n$ vertices and suppose that at least one of the following holds.
\begin{itemize}
\item[{\rm  (i)}] $\delta(G) \geq (\bigdelta + \epsilon)n$, where $\bigdelta := \max\{\delta^*_{K_{\chi(F)}}, 1-1/6e(F)\}$ and $\chi(F)$ is the chromatic number of $F$.
\item[{\rm (ii)}] $F$ is $d$-regular and $\delta(G) \geq (\bigdelta + \epsilon)n$, where $\bigdelta := \max\{\delta^*_{K_{\chi(F)}}, 1-1/3d\}$.
%\item[{\rm (iii)}] $F=K_3$ and $\delta(G) \geq  (\delta^*_{K_3} + \eps)  n$.
\item[{\rm (iii)}] $F=C_\ell$, where $\ell\ge 3$ is odd, and $\delta(G) \geq  (\delta^*_{C_\ell} + \eps)  n$.
\end{itemize}
Then $G$ has an $F$-decomposition.
\end{theorem}
Furthermore, asymptotically optimal results for even cycles have been obtained in~\cite{BKLO}
and for all bipartite graphs with a leaf by Yuster~\cite{bipyuster}.
Note that by Theorem~\ref{exact-decompositions}(iii)
it suffices to show that $\delta^*_{K_3} \le 3/4$ in order to prove Conjecture~\ref{conjecture} asymptotically.
Determining $\delta^*_{K_r}$ is therefore an important problem, as well as being interesting in its own right. 
The best current result towards the triangle case is due to Dross~\cite{dross}, who gave a very short and elegant argument showing that $\delta^*_{K_3} \leq 0.9$. This improves previous bounds of Yuster~\cite{yuster2005asymptotically}
Dukes~\cite{dukes,dukes2} and Garaschuk~\cite{garaschuk2014linear}.
For $r\ge 4$, Yuster~\cite{yuster2005asymptotically} proved that $\delta^*_{K_r} \leq 1-1/9r^{10}$; 
this was subsequently improved by Dukes~\cite{dukes,dukes2} who showed that $\delta^*_{K_r}\leq 1-2/9r^2(r-1)^2$.
On the other hand, a construction showing $\delta^*_{K_r}\geq 1-1/(r+1)$ is described in \cite{yuster2005asymptotically}.
Our main result gets substantially closer to this lower bound for large $r$.

\begin{theorem} \label{fracdecomp}
The following holds for any integers $r\geq 3$ and $n\geq 10^4r^3$. 
If~$G$ is a graph on~$n$ vertices and $\d(G)\geq (1-1/10^4r^{3/2})n$, then~$G$ has a fractional $K_r$-decomposition.
\end{theorem}
In order to clarify the presentation, we have made no attempt to optimise the constant $10^4$ appearing in Theorem~\ref{fracdecomp}.
Along the way, we also obtain a comparatively short and simple proof that $\d(G)\geq (1-1/10^5r^{2})n$ 
guarantees a fractional $K_r$-decomposition (see Theorem~\ref{r2lemma}).

Together with Theorem~\ref{exact-decompositions}, we immediately obtain the following corollary.
Note that (iii) is a special case of (ii).%
\COMMENT{the constant in (i) is rather slack, so no need for an $\eps$.
could replace d+1 by d in (ii) if use Brooks theorem and state Theorem~\ref{exact-decompositions}(iii) for odd cycles, but it doesn't seem worth it }
\COMMENT{
An infinite family of graphs with $\delta(G) = (1-1/(r+1))|G| - 1$ can be defined as follows.
For each $s \in \mathbb N$, let $m = 2s(r+1)$ and let $H_s$ be the complete $(r-1)$-partite graph with vertex classes of size $m$.
Let $G_s$ be a graph obtained from $H_s$ by adding a $(4s-1)$-regular graph inside each class of $H_s$.
Then $G$ is $d$-regular with $d = (1-1/(r+1))|G_s| - 1$ and every copy of $K_r$ in $G_s$ contains an edge inside one of the parts of $H_s$.
But less than a $1/\binom r 2$ fraction of the edges of $G_s$ are inside a class of $H_s$, so $G_s$ can have no fractional $K_r$-decomposition.
}
\begin{corollary} \label{Fcombined}
Let $F$ be a graph, let $\epsilon > 0$ and let $n$ be sufficiently large.
Let $G$ be an $F$-divisible graph on $n$ vertices such that at least one of the following holds.
\begin{itemize}
\item[{\rm (i)}] $\delta(G) \geq (1-1/10^4|F|^2)n$.
\item[{\rm (ii)}] $F$ is $d$-regular and $\delta(G) \geq (1-1/10^4(d+1)^{3/2} + \epsilon)n$. 
\item[{\rm (iii)}] $F=K_r$ and $\delta(G) \geq  (1-1/10^4r^{3/2}+\eps)n$.
\end{itemize}
Then $G$ has an $F$-decomposition.
\end{corollary}
An obvious open problem is to improve the bounds in Theorem~\ref{fracdecomp} (and thus in Corollary~\ref{Fcombined}).
Furthermore, in view e.g.~of Theorem~\ref{exact-decompositions}(iii) it would also be very interesting to obtain better bounds on
the fractional decomposition threshold for odd cycles.

\subsection{(Fractional) decompositions of hypergraphs}
Our methods also extend to $k$-uniform hypergraphs with $k \ge 3$.  
For a $k$-uniform hypergraph $G$, the \defn{minimum codegree} $\delta_{k-1}(G)$ of $G$ is the minimum over all $(k-1)$-subsets $S$ of $V(G)$ of the number of edges containing all the vertices in $S$.
For a $k$-uniform hypergraph $F$ and $n \in \mathbb N$, let $\delta^*_F(n)$ be the infimum over all $\bigdelta$ such that every $k$-uniform hypergraph $G$ on $n$ vertices with $\delta_{k-1}(G) \geq \bigdelta n$ has a fractional $F$-decomposition.
We again call $\delta^*_F := \limsup_{n \to \infty} \delta^*_F(n)$ the \defn{fractional $F$-decomposition threshold}.
For $r \ge k \ge 2$, let $\clique k r$ denote the complete $k$-uniform hypergraph on $r$ vertices.
For $r \ge k \ge 2$, Yuster~\cite{yuster2005fractional} proved that $\delta^*_{\clique k r} \leq 1 - 1/6^{kr}$.
Dukes~\cite{dukes,dukes2} improved this to $\delta^*_{\clique k r} \leq 1 - 1/(2\cdot 3^k\binom r k ^2)$.
We give a short combinatorial proof for a similar bound (which is slightly better when $r$ is large).

\begin{theorem} \label{hypergraphs}
Given $r, k\in \mathbb N$ with $r > k \ge 2$, let $\delta:=\frac{k!}{2^{k+3}k^2r^{2k-1}}$ and let $n > 1/\delta$.%
\COMMENT{Made explicit for small $r$ graph purposes.
Note that this condition is trivial for graphs, as for smaller $n$ the minimum degree condition implies that $G$ is complete.}
Then any $k$-uniform hypergraph $G$ on $n$ vertices with $\delta_{k-1}(G) \geq (1-\delta)n$ has a fractional $\clique k r$-decomposition.
\end{theorem}
Note that in Theorem~\ref{hypergraphs} and at many points in the remainder of the paper, we use $\delta$ for the `missing minimum degree proportion' (rather than for the minimum degree of the (hyper-)graph itself). Moreover,
note that for graphs, Theorem~\ref{hypergraphs} gives weaker bounds than those discussed in the previous subsection.

%For $k\ge 3$ we are aware of no lower bounds on the minimum codegree required to guarantee a fractional
%$\clique k r$-decomposition beyond the codegree threshold for a single copy of $\clique k r$ to appear.
%Lo and Markstr\"om~\cite{hypercodeg} gave a construction that showed that this is at least $(1-1/(r-k+1))n$.
%(The same construction was found independently by Falgas-Ravry~\cite{victor} in the case when $k=3$.)

In a recent breakthrough, Keevash~\cite{keevash} proved that every sufficiently large $K^{(k)}_n$ satisfying the necessary 
divisibility conditions has a $K^{(k)}_r$-decomposition. 
This settled a question regarding the existence of designs going back to the 19th century.
Moreover, his results also extend to hypergraphs with minimum codegree at least $(1-\eps)n$, for an unspecified $\eps>0$.
Theorem~\ref{hypergraphs} may help to obtain explicit bounds on $\eps$.

\subsection{Recent developments}
Since submission of the original manuscript, there have been a number of 
further developments:
Firstly, Glock, K\"uhn, Montgomery, Lo and Osthus~\cite{GKLMO} obtained further results 
on the decomposition threshold of graphs which strengthen Theorem~\ref{exact-decompositions}. 
These imply e.g. that for cliques, the decomposition threshold equals its fractional
version, i.e.~a minimum degree of $(\delta^*_{K_r}+o(1))n$ guarantees a 
$K_r$-decomposition. Also,~\cite{GKLMO} determines the decomposition threshold
for bipartite graphs.

Secondly, several results in the partite setting have been obtained.
Bowditch and Dukes~\cite{Du} and Montgomery~\cite{Montgomery} gave bounds on the fractional partite decomposition 
threshold of triangles and cliques respectively, and Barber, K\"uhn, Lo, Osthus and Taylor~\cite{BKLOT} 
showed that the partite decomposition threshold for cliques equals its
fractional version.
These results can be combined to show the existence of completions of 
suitable partial 
(mutually orthogonal) latin squares.

Thirdly, Glock, K\"uhn, Lo and Osthus~\cite{GKLO} obtained a new proof of the existence of
designs which generalizes the results in~\cite{keevash} beyond the quasirandom 
setting. In particular, the result implies that a minimum codegree bound
close to that in Theorem~\ref{hypergraphs} already guarantees an actual decomposition (under appropriate divisibility conditions).
The proof in~\cite{GKLO} incorporates the ideas used in the proof of 
Theorem~\ref{hypergraphs}.

\subsection{Proof idea and organization of the paper.}
The proof by Dukes~\cite{dukes,dukes2} that $\delta^*_{K_r}\leq 1-\Omega(1/r^4)$ is based on tools from linear algebra. 
To prove Theorem~\ref{fracdecomp} we build on the combinatorial approach of Dross~\cite{dross}.
The latter argument begins with a uniform weighting of the triangles in a graph $G$ with high minimum degree
(this idea is actually already implicit in~\cite{dukes}).
This uniform weighting can be shown to be `close' to a fractional triangle decomposition of $G$. 
Then the idea is to use the  max-flow min-cut theorem to make the necessary adjustments to this weighting 
to obtain a fractional triangle decomposition.
Our methods begin with a similar initial weighting, but avoid using the  max-flow min-cut theorem.
% to make global adjustments by using different gadgets to make local adjustments to the weight over each edge.
%For triangles, this approach is less efficient than that used by Dross, but it is easier to analyse for larger cliques. 
%This enables us to push the initial bound attained by these methods ($\delta^*_{K_r}\leq 1-\Omega(r^{-3})$) further, and prove Theorem~\ref{fracdecomp}.
Theorem~\ref{hypergraphs} is obtained by generalising (simplified versions of) these methods to hypergraphs;
we give a more detailed sketch in Section~\ref{sketch}. 
We then prove Theorem~\ref{hypergraphs} in Section~\ref{sec:hyper}, before proving Theorem~\ref{fracdecomp} in Sections~\ref{sec:num-cliques}, \ref{secweight} and~\ref{sec:vertexgadget}--\ref{secproof}.
In Section~\ref{escape-route} we combine the results of Sections~\ref{sec:num-cliques} and~\ref{secweight} to give a short proof of Theorem~\ref{r2lemma}%
\COMMENT{B: Renamed Lemma to Theorem, as we can't call a final result a lemma.}, a weaker form of Theorem~\ref{fracdecomp} with $r^2$ in place of $r^{3/2}$.

Our argument here and that in~\cite{BKLO} is purely combinatorial.
So the proofs of Theorem~\ref{fracdecomp} and Theorem~\ref{exact-decompositions} together  yield a combinatorial proof of 
Wilson's theorem~\cite{wilson1,wilson2,wilson3,Wilson} that every large $F$-divisible clique has an $F$-decomposition.
(The original proof as well as that of the clique version for hypergraphs by Keevash~\cite{keevash} made use of algebraic tools.)

\subsection{Notation}
Given $k \ge 2$, a \emph{$k$-uniform hypergraph} is an ordered pair $G = (V(G),E(G))$, where $V(G)$ is a finite set (the vertex set) and $E(G)$ is a set of $k$-element subsets of~$V(G)$ (the edge set).
We sometimes write $V(e)$ to emphasise that we are thinking of an edge $e$ as a set of vertices.
Given a $k$-uniform hypergraph $G$ and $S \subseteq V(G)$ with $|S| \leq k-1$, we let $N(S):=\{T \subseteq V(G) \setminus S : T \cup S \in E(G)\}$ and write $d(S):=|N(S)|$.
We let $N^c(S) := \{T \subseteq V(G) : |T|=k-|S|, T \cup S \notin E(G)\}$.
For $1 \leq j \leq k-1$, we write $\delta_j(G) := \min \{d(S) : S \subseteq V(G), |S| = j\}$ for the \defn{minimum $j$-degree} ($\delta_{k-1}(G)$ is also known as the \defn{minimum codegree}).

Given $r \geq k\ge 2$, we write $\cliques k r (G)$ for the set of copies of $\clique k r$ in~$G$.
If~$G$ is clear from the context, we just write $\cliques k r$; if $k=2$, then we just write $\KK_r$.
We write $k_r = k_r(G) := |\cliques k r (G)|$ for the number of $r$-cliques in $G$.
(For $r < 0$, we let $k_r:=0$.)%
\COMMENT{We do use that $k_r = \binom n r$ for $0 \leq r \leq k$.}
For each $S \subseteq V(G)$ and $r\in \N$, let $\kappa_S^{(r)} := |\{ K \in \cliques k r : S \subseteq V(K) \}|$.
For an edge $e$, we often write $\kappa_e^{(r)}$ for $\kappa_{V(e)}^{(r)}$.
For $r,k\in\mathbb{N}$, we write $(r)_k:=r(r-1)\cdots (r-k+1)$ for the $k$th falling factorial of~$r$.
%For $a, b, c \in \mathbb R$ we write $a = b \pm c$ if $b-c \leq a \leq b+c$.

For a graph $G$ and $x \in V(G)$, we write $N(x) := \{y \in V(G) : xy \in E(G)\}$ for the neighbourhood of $x$ and $d(x) := |N(x)|$ for the degree of $x$.
We let $N^c(x) = \{y \in V(G) : xy \notin E(G)\}$ (note that this includes $x$ itself).
For $S \subseteq V(G)$, we write $G[S]$ for the subgraph of $G$ induced by~$S$, and abbreviate $e(G[S])=|E(G[S])|$ and $\KK_r(G[S])$ by $e(S)$ and $\KK_r[S]$, respectively. 
Given any event $A$, we let
\[
\mathbf{1}_A :=\begin{cases}
1 & \text{ if }A \text{ occurs,} \\
0 & \text{ otherwise.}
\end{cases}
\]

By a \defn{weighting} of the $r$-cliques in $G$ we mean a function $\omega : \cliques k r \to \mathbb R$.
The \defn{weight} of a clique $K$ is $\omega(K)$.
For $e \in E(G)$, the \defn{weight over $e$} is $\sum_{K \in \KK_r(G) : e \in E(K)} \omega(K)$.%
\COMMENT{This is mostly for the benefit of the sketch: the word weight does not seem to appear much in the body of the proof, as we usually work with functions.}

\section{Sketch of proof} \label{sketch}

Here we present a sketch proof of Theorem~\ref{hypergraphs}, which will also form the backbone of the proof of Theorem~\ref{fracdecomp}.
For simplicity, we describe the argument for graphs (which generalises straightforwardly to the hypergraph case).

As $\delta(G)$ is large, for each $e \in E(G)$, $G$ has many $r$-cliques containing $e$.
In fact, all edges $e$ are contained in approximately the same number of $r$-cliques.
%$G$ is missing very few $r$-cliques containing $e$ from the complete graph on the same set of vertices.
%All $e$ are therefore in almost the maximum possible number of $r$-cliques; in particular, all $e$ are in approximately the same number of $r$-cliques.
More precisely, there is some small $\alpha >0$ such that $(1- \alpha)k_{r-2}  \le \kappa_{e}^{(r)} \le k_{r-2}$ for any $e \in E(G)$
(see Proposition~\ref{global-and-local-cliques}).
%f\eqref{cliques-over-edge}.)
An appropriately scaled uniform weighting of the $r$-cliques of $G$ is therefore already close to a fractional decomposition of $G$, in the sense that the total weight over each edge is close to $1$.
We seek to perturb the weight of each $r$-clique so that the total weight over each edge becomes exactly $1$.

For each $e \in E(G)$, we consider an `edge-gadget' $\psi_e$ that permits us to alter the weight over $e$ without altering the weight over any other edge.
This edge-gadget adds weight to some $r$-cliques and removes weight from other $r$-cliques so that the change in weight cancels out over every edge except for~$e$.
Formally speaking, an \emph{edge-gadget} for the edge $e$ is a weighting $\psi_e : \KK_r \to \R$ such that, for each $f \in E(G)$,
\begin{align*}
\sum_{ K\in \KK_r \colon  f \in E(K)} \psi_e(K)=  \mathbf{1}_{\{e=f\}}. 
\end{align*}
For $c \in \R$, the function $c \cdot \psi_e$ corresponds to adding weight $c$ over~$e$. 
Our aim is to use these edge-gadgets $\psi_e$ (for $e \in E(G)$) to correct the weights over the edges without reducing the weight of any one clique so far that it becomes negative.

We construct a basic edge-gadget~$\psi_e$ as follows.
Let $J$ be an $(r+2)$-clique of $G$ that contains $e$ (which exists since the minimum degree is large).
There are three types of edges in $J$, determined by how many vertices they share with $e$.
Accordingly, write $E_j := \{f \in E(J) : |V(f) \cap V(e)| = j\}$.
(Note that $E_2 = \{e\}$.)
Similarly, there are three types of cliques in $\KK_r(J)$, determined by how many vertices they share with $e$.
Accordingly, write $S_j := \{K \in \KK_r(J) : |V(K) \cap V(e)| = j\}$. 
We first increase the weight of every $r$-clique in $S_2$ by $1/|S_2| = 1/ \binom r 2$.
This has the effect of increasing the weight over every edge of $J$, and this increase only depends on whether the edge is in $E_2$, $E_1$ or $E_0$.
The weight over $e$ is now $1$ as desired, but the weight over the edges in $E_1$ and $E_0$ is also positive.
We now correct the weight over every edge in $E_1$ by reducing the weight of every clique in $S_1$ by the same amount.
The weight over each edge in $E_2 \cup E_1$ is now as desired, so it remains only to correct the weights of each edge in $E_0$.
But the edges in $E_0$ form a clique $K^*$, and the weight over every edge in $E_0$ is identical, so it can be made equal to zero by adjusting the weight of $K^*$.
This completes the construction of $\psi_e$.

If we use the basic edge-gadget $\psi_e$ as described above to adjust the weight over each edge~$e$, then we might have to make large adjustments to the weights of some cliques---large enough that these weights would become negative and prevent us obtaining a fractional decomposition. 
To avoid making too large an adjustment to the weight of any $r$-clique, we will therefore, for each edge~$e$, use many different edge-gadgets $\psi_e$ to correct the weight over $e$, making a small adjustment using each $\psi_e$ and spreading the adjustments over as many $r$-cliques as possible. 
To be precise, note that in the previous paragraph we have actually defined an edge-gadget, $\psi_{e}^{J}$ say, for each $(r+2)$-clique $J$ containing $e$, and there are $\kappa_e^{(r+2)}$ such cliques $J$.
So we set $\psi_e$ to be the average over of the edge-gadgets $\psi_{e}^{J}$, that is, $\psi_e : = \sum_{J \in \KK_{r+2}: e \in E(J)} \psi_{e}^{J} / \kappa_e^{(r+2)} $.

This simple argument can already be used to find fractional $K_r$-decompositions of graphs on $n$ vertices with minimum degree at least $(1-c/r^3)n$ for some absolute constant $c$.
The argument generalises straightforwardly to hypergraphs, and we use it to prove Theorem~\ref{hypergraphs} in Section~\ref{sec:hyper}.

In order to prove Theorem~\ref{fracdecomp}, we introduce two additional ideas. 
Firstly we introduce an additional preprocessing step which allows us to limit the adjustments we need to make to the weight over most of the edges.
This leaves us most concerned with the problem of correcting the weight over a small fraction of `bad' edges.
The naive averaging argument would then ask for a large adjustment to the weight of cliques that contain many bad edges. 
But the proportion of such gadgets which use many bad edges is small.
Hence we can avoid using these gadgets and thus reduce the maximum adjustment that might be required for each $r$-clique. 
This allows us to obtain fractional $K_r$-decompositions provided that $\delta(G) \geq (1-c/r^2)n$ for some absolute constant $c$.
We prove this in Sections~\ref{sec:num-cliques}--\ref{escape-route}.

Secondly, we introduce a `vertex-gadget' that allows us to increase the weight over every edge at a vertex by the same amount simultaneously (see Section~\ref{sec:vertexgadget}).
In return for this reduction in flexibility we are able to make these adjustments more efficiently, with smaller changes to the weights of cliques. 
By further analysing the pattern of changes required to the weights over the edges in Section~\ref{sec:kappa}, we use this vertex-gadget to make an initial adjustment before using edge-gadgets to make the final adjustment. 
We put together these ideas and results from Sections~\ref{sec:num-cliques}, \ref{secweight}, \ref{sec:vertexgadget} and~\ref{sec:kappa} to prove Theorem~\ref{fracdecomp} in Section~\ref{secproof}.

\section{Fractional decompositions of hypergraphs}\label{sec:hyper}

\subsection{Basic tools} \label{hyper-tools}

We first observe that if a $k$-uniform hypergraph has large minimum codegree, then for all $\ell<k$ its
minimum $\ell$-degree is also large.

\begin{proposition} \label{other-degrees}
Let $k \in \mathbb N$ with $k \geq 2$, let $0<\delta<1$ and let $G$ be a $k$-uniform hypergraph on $n$ vertices.
Suppose that $\delta_{k-1}(G) \geq (1 - \delta) n$.
Then for every $\ell \leq k-1$, $\delta_\ell(G) \geq (1 - \delta) \binom {n-\ell} {k - \ell}$.
\end{proposition}

\begin{proof}
Choose $S \subseteq V(G)$ with $|S| = \ell$ such that $d(S) = \delta_\ell(G)$.
Let $\mathcal T := \{(T, e) : S \subseteq T \subseteq V(e), e \in E(G), |T| = k-1 \}$.
Then
\[
\binom {n-\ell} {k-\ell-1} \cdot \delta_{k-1}(G) \leq |\mathcal T| = d(S) \cdot (k - \ell),
\]
hence
\[
\delta_\ell(G) = d(S) \geq \frac {(1-\delta)n} {k-\ell} \binom {n-\ell} {k-\ell-1} \geq (1-\delta) \binom {n-\ell} {k-\ell}. \qedhere
\]
\end{proof}

We shall use the following bounds on the number of $r$-cliques and the number of $r$-cliques containing a fixed edge.

\begin{proposition} \label{global-and-local-cliques}
Let $n >r >k \geq 2$, let $1/n < \delta < 1$ and let $G$ be a $k$-uniform hypergraph on $n$ vertices with $\delta_{k-1}(G) \geq (1-\delta)n$.
Then
\begin{equation} \label{number-of-cliques} 
\Big(1 - \binom r k \delta\Big) \binom n r \leq k_r \leq \binom n r \leq \frac {n^r} {r!}
\end{equation}
and, for any $e \in E(G)$,
\begin{align} \label{cliques-over-edge}
k_{r-k} - \frac{2 k\delta n^{r-k}}{(r-k)!}\binom{r}{k-1} \leq \extensions e r \leq k_{r-k}.
\end{align}
\end{proposition}

\begin{proof}
We first prove~\eqref{number-of-cliques}.
The upper bound is clear.
To see the lower bound, consider constructing a clique one vertex at a time.
Since each new vertex must form an edge with all $(k-1)$-subsets of the previously chosen vertices, the number of $r$-cliques is at least%
\COMMENT{Last step is equality as there is no contribution for smaller $s$.}
\begin{align*} 
& (n)_{k-1} \cdot (1-\delta)n \cdot (1 - k\delta)n \cdot (1 - \textstyle \binom {k+1} {k-1} \delta) n \cdots (1 - \textstyle \binom {r-1} {k-1} \delta) n /r! \\
& \qquad \qquad \geq (1-\sum_{s=k}^r \textstyle \binom {s-1} {k-1} \delta) (n)_r / r! = (1 - \binom r k \delta) \binom n r.
\end{align*}

We now verify~\eqref{cliques-over-edge}.
We have that $\extensions e r = k_{r-k} - g(e)$, where $g(e)$ is the number of $K \in \cliques k {r-k}$ such that $V(e) \cup V(K)$ does not induce an $r$-clique in $G$.
This happens when either $V(e) \cap V(K) \neq \emptyset$, or when there is a non-edge $f$ of $G$ contained in $V(e) \cup V(K)$.
The number of $K \in \cliques k {r-k}$ with $V(e) \cap V(K) \neq \emptyset$ is at most $k \cdot k_{r-k-1}$.
And for a fixed non-edge $f$ of $G$, the number of $K \in \cliques k {r-k}$ such that $V(e) \cap V(K) = \emptyset$ and $V(f) \subseteq V(e) \cup V(K)$
is at most $k_{r-k-|V(f) \setminus V(e)|}$ (which is $0$ if $r < k + |V(f) \setminus V(e)|$).
Thus%
   \COMMENT{In the final line below we use that $\binom{k}{k-j}=\frac{j+1}{k-j}\binom{k}{k-1-j}$.
	In the second line the first term is less that the $j=1$ term of the sum as $1 \leq \delta n$.}
\begin{align*}
g(e) & \leq k \cdot k_{r-k-1} + \sum_{j=1}^{k-1} \sum_{S \subseteq V(e) : |S| = k-j} |N^c(S)| k_{r-k-j} \\
     & \leq \frac {k n^{r-k-1}} {(r-k-1)!} + \sum_{j=1}^{\min \{k-1, r-k\}} \binom k {k-j} \cdot \frac {\delta n^j} {j!} \cdot \frac {n^{r-k-j}} {(r-k-j)!} \\
     & \leq 2 \delta n^{r-k} \sum_{j=1}^{\min \{k-1, r-k\}} \frac {\binom k {k-j}} {j! (r-k-j)!}
		=\frac{2 \delta n^{r-k}}{(r-k)!}\sum_{j=1}^{k-1} \binom{k}{k-j}\binom{r-k}{j}\\
		& \le \frac{2 k\delta n^{r-k}}{(r-k)!}\sum_{j=0}^{k-1} \binom{k}{k-1-j}\binom{r-k}{j} = \frac{2 k\delta n^{r-k}}{(r-k)!}\binom{r}{k-1},
\end{align*}
where the second inequality uses \eqref{number-of-cliques} and Proposition~\ref{other-degrees}.
\end{proof}

Let $G$ be a $k$-uniform hypergraph. 
An \defn{edge-weighting} of $G$ is a function $\omega : E(G) \to \mathbb R$.
For the rest of this section, it will be convenient to view the set of edge-weightings of $G$ as an $e(G)$-dimensional vector space $\Omega(G)$.
This space has a natural basis $\{\indicator_e : e \in E(G)\}$, where
\[
\indicator_e(f) := \begin{cases}
1 & \text{if } e = f, \\
0 & \text{otherwise.}
\end{cases}
\]
We shall identify $\indicator_e$ with $e$ itself, and sums of edges with the corresponding subgraphs of $G$;
thus we write $H = \sum_{e \in E(H)} e$ for every subgraph $H$ of $G$.
Let $\Omega_r(G) := \{\sum_{K \in \cliques k r} \omega(K) K : \omega(K) \in \mathbb R\}$ be the subspace of $\Omega(G)$ spanned by the $r$-cliques of $G$ and
let $\Omega^+_r(G) := \{\sum_{K \in \cliques k r} \omega(K) K : \omega(K) \geq 0\}$.
We claim that if $G \in \Omega^+_r(G)$, then $G$ has a \defn{fractional $K_r^{(k)}$-decomposition}.
Indeed, observe that, if $G \in \Omega^+_r(G)$, then there is an $\omega : \cliques k r \to \mathbb R_{\geq 0}$ such that
\begin{align*}
\sum_{e \in E(G)} e & = G 
    = \sum_{K \in \cliques k r} \omega(K) K 
    = \sum_{K \in \cliques k r} \omega(K) \sum_{e \in E(K)} e 
	  = \sum_{e \in E(G)} \Big( \sum_{K \in \cliques k r : e \in E(K)} \omega(K) \Big) e;
\end{align*}
%that is, the total weight of the $r$-cliques containing each edge is $1$.
that is, the weight over each edge is exactly $1$.
Moreover, since no weight is negative it must be the case that $\omega$ is a function from $\cliques k r (G)$ to $[0,1]$.

\subsection{Adding weight over an edge}

We now describe the basic edge-gadget that allows us to increase or decrease weight over a single edge by adjusting the weights of a suitable set of $r$-cliques.

\begin{proposition} \label{gadget}
Let $r > k$. There are $\alpha_0,\dots,\alpha_k \in \mathbb R$ so that the following holds.
Let $J$ be a copy of $\clique k {k+r}$ and let $e \in E(J)$.
Then the weighting $\omega : \cliques k r (J) \to \mathbb R$ defined by $\omega(K) := \alpha_{|V(e) \cap V(K)|}$ satisfies
\begin{itemize}
	\item[\rm(i)] $e = \sum_{K \in \cliques k r (J)} \omega(K) K$,
	\item[\rm(ii)] if $|V(e) \cap V(K)| = i$, then $ |\omega(K)| = |\alpha_i| \leq \frac {2^{k-i} (k-i)!} {\binom {r-k+i} i}$.
\end{itemize}
\end{proposition}

As discussed in Section~\ref{sketch}, the idea of the proof is that to increase the weight over the edge $e$, we first increase the weight of every $r$-clique containing $e$.
This puts too much weight over the edges that share $k-1$ vertices with $e$, so we remove this weight by decreasing the weight on each $r$-clique that shares $k-1$ vertices with $e$.
Continuing in this fashion we eventually obtain a (signed) weighting of the $r$-cliques of $J$ such that the net weight over each edge $f$ is non-zero if and only if $e = f$.

\begin{proof}[Proof of Proposition~\ref{gadget}]
Let 
$$
\Omega^e(J) := \{\sum_{f \in E(J)} \omega_f f : \omega_{f_1} = \omega_{f_2} \text{ if } |V(e) \cap V(f_1)| = |V(e) \cap V(f_2)|\}
$$ 
be the $(k+1)$-dimensional subspace of $\Omega(J)$ in which the weight of each edge depends only on the size of its intersection with $e$.
For $0 \leq i \leq k$, let $E_i := \sum_{f \in E(J) : |V(e) \cap V(f)| = i} f$.
The $E_i$ are a natural basis for $\Omega^e(J)$.
For $0 \leq j \leq k$, let $F_j := \sum_{K \in \cliques k r (J) : |V(e) \cap V(K)| = j} K$.
We claim that the $F_j$ also form a basis for $\Omega^e(J)$.
To see this, we first calculate $F_0, \ldots, F_k$ in terms of $E_0, \ldots, E_k$.
Write $F_j = \sum_{i=0}^k a_{ij} E_i$.
Then $a_{ij}$ is the number of ways to extend an edge meeting $e$ in $i$ vertices to an $r$-clique meeting $e$ in $j$ vertices.
We need an additional $j-i$ vertices from $e$ and an additional $r-j-(k-i)$ vertices from outside $e$, so
\begin{equation} \label{matrix-coefficients}
a_{ij} = \binom {k-i} {j-i} \binom {r - (k-i)} {r - j - (k-i)} = \binom {k-i} {j-i} \binom {r-k+i} j. 
\end{equation}
In particular, $a_{ij} = 0$ for $i > j$ and $a_{ij} \neq 0$ when $i = j$, so the matrix $(a_{ij})$ is upper triangular with non-zero diagonal entries and hence is invertible. 
Thus there exist $\alpha_0, \ldots, \alpha_k$ such that
\begin{align} 
e = E_k & = \sum_{j=0}^k \alpha_j F_j\label{change-of-basis0}\\
& = \sum_{j=0}^k \alpha_j \sum_{i=0}^k a_{ij} E_i = \sum_{i=0}^k \big(\sum_{j=0}^k a_{ij} \alpha_j\big) E_i=\sum_{i=0}^k \big(\sum_{j=i}^k a_{ij} \alpha_j\big) E_i.\label{change-of-basis}
\end{align}
Set $\omega(K) = \alpha_{|V(e) \cap V(K)|}$. 
Then (\ref{change-of-basis0}) proves (i). 
To see (ii) first note that, by \eqref{matrix-coefficients} and (\ref{change-of-basis}), 
$\alpha_k = 1/a_{kk}=1 / \binom r k$, and, for $0\le i<k$, 
\begin{equation}\label{eq:irow}
\sum_{j=i}^k \binom {k-i} {j-i} \binom {r-k+i} j \alpha_j =\sum_{j=i}^k a_{ij} \alpha_j= 0.
\end{equation}
We shall prove by induction on $k-i$ that $|\alpha_i| \leq 2^{k-i} (k-i)! / \binom {r-k+i} i$. 
This holds with equality for $\alpha_k$, so assume that $0\le i < k$. Then by induction,%
\COMMENT{To see the 3rd inequality use that the RHS of the first lines is equal to
	$\frac {2^{k-i} (k-i)!} {\binom {r-k+i} i} \sum_{j=i+1}^k \frac {\binom {k-i} {j-i} \binom {r-k+i} j} {2^{j-i} (k-i)_{j-i} \binom {r-k+j} j}$}
\begin{align*}
|\alpha_i| & \stackrel{(\ref{eq:irow})}{\leq} \sum_{j=i+1}^k \frac {\binom {k-i} {j-i} \binom {r-k+i} j} {\binom {r-k+i} i} |\alpha_j| 
             \leq \sum_{j=i+1}^k \frac {\binom {k-i} {j-i} \binom {r-k+i} j} {\binom {r-k+i} i} \frac {2^{k-j} (k-j)!} {\binom {r-k+j} {j}}
%         \\& =    \frac {2^{k-i} (k-i)!} {\binom {r-k+i} i} \sum_{j=i+1}^k \frac {\binom {k-i} {j-i} \binom {r-k+i} j} {2^{j-i} (k-i)_{j-i} \binom {r-k+j} j}
			 	 \\& \,\,\,\leq \frac {2^{k-i} (k-i)!} {\binom {r-k+i} i} \sum_{j=i+1}^k \frac 1 {2^{j-i} (j-i)!}
				     \leq \frac {2^{k-i} (k-i)!} {\binom {r-k+i} i},
\end{align*}
as required.
\end{proof}

\subsection{Proof of Theorem~\ref{hypergraphs}}

We are now ready to put everything together to prove Theorem~\ref{hypergraphs}.

\begin{proof}[Proof of Theorem~\ref{hypergraphs}]
Let $\kappa := \sum_{e \in E(G)} \extensions e r / e(G)$ be the average value of $\extensions e r$, and let $w := 1/\kappa$.
By \eqref{cliques-over-edge} of Proposition~\ref{global-and-local-cliques},
\begin{equation} \label{top-right}
|\kappa-\extensions e r| \leq \frac{2 k\delta n^{r-k}}{(r-k)!}\binom{r}{k-1}.
\end{equation}
Observe also that (in $\Omega(G)$)
\begin{equation} \label{uniform-weighting}
\sum_{K \in \cliques k r (G)} K = \sum_{e \in E(G)} \extensions e r e.
\end{equation}
By Proposition~\ref{gadget}, for every $K \in \cliques k r (G)$ and every $e \in E(G)$, there exists $\omega_K^e$ with
\begin{equation} \label{top-left}
|\omega_K^e| \leq \frac {2^{k-j} (k-j)!} {\binom {r-k+j} j}, \ \ \ \text{where } j=|V(e)\cap V(K)|,
\end{equation}
and such that for every $J \in \cliques k {r+k} (G)$ with $e \in E(J)$, 
\begin{equation} \label{gadget-coefficients}
e = \sum_{K \in \cliques k r (J)} \omega_K^e K.
\end{equation}
Thus
\begin{align*}
G & = \sum_{e \in E(G)} \kappa we = \sum_{e \in E(G)} (\extensions e r w e + (\kappa - \extensions e r) w e) \\
  & \stackrel{\mathclap{\eqref{uniform-weighting}, \eqref{gadget-coefficients}}} 
	  = \sum_{K \in \cliques k r (G)} w K + \sum_{e \in E(G)} \frac {(\kappa-\extensions e r)w} {\extensions e {r+k}} \sum_{(J \in \cliques k {r+k} (G) : e \in E(J))} \sum_{K \in \cliques k r (J)} \omega_K^e K \\
  & = \sum_{K \in \cliques k r (G)} \Big(w  + \sum_{(J \in \cliques k {r+k}(G) : K \subseteq J)} \sum_{e \in E(J)} \frac {\omega_K^e(\kappa-\extensions e r)w} {\extensions e {r+k}}\Big)K,
\end{align*}
and it suffices to show that $w  + \sum_{J \in \cliques k {r+k}(G) : K \subseteq J} \sum_{e \in E(J)} \frac {\omega_K^e(\kappa-\extensions e r)w} {\extensions e {r+k}} \geq 0$ for every $K \in \cliques k r(G)$. 
(Indeed, then $G\in \Omega^+_r(G)$ and so $G$ has a fractional $K^{(k)}_r$-decomposition by our remarks at the end of Section~\ref{hyper-tools}.)
So fix $K \in \cliques k r (G)$, let $J \in \cliques k {r+k} (G)$ with $K \subseteq J$ and let $e \in E(J)$.
By Proposition~\ref{global-and-local-cliques},
\begin{align} \label{bottom}
\extensions e {r+k} & \geq (1 - \textstyle \binom r k \delta) (n)_r / r! - \frac{2 k\delta n^{r}}{r!}\binom{r+k}{k-1} \geq \frac{n^r}{2 r!},
\end{align}
say, with plenty of room to spare.
Now by \eqref{top-right}, \eqref{top-left} and \eqref{bottom}, if $j=|V(e)\cap V(K)|$ then%
   \COMMENT{to see the last inequality note that the second expression below is equal to
	$	\frac{2^{k-j+2}  k \delta}{n^k}\frac{(k-j)!j!(r-k)!r!(r)_{k-1}}{(r-k+j)!(r-k)!(k-1)!}=\frac{2^{k-j+2}  k^2 \delta}{n^k}\frac{(r)_{k-j}(r)_{k-1}}{\binom{k}{j}}.$}
\begin{align*}
\left| \frac {\omega_K^e(\kappa-\extensions e r)} {\extensions e {r+k}} \right| 
& \leq 
\frac{| \omega_K^e | |\kappa-\extensions e r |}{\extensions e {r+k}}
\leq
\frac {\frac {2^{k-j} (k-j)!} {\binom {r-k+j} j} \cdot 
        \frac{2 k\delta n^{r-k}}{(r-k)!}\binom{r}{k-1}} {n^r / 2 r!} 
  %= \frac{2^{k-j+2}  k \delta}{n^k}\frac{(k-j)!j!(r-k)!r!(r)_{k-1}}{(r-k+j)!(r-k)!(k-1)!} 
	\\
  & =\frac{2^{k-j+2}  k^2 \delta}{n^k}\frac{(r)_{k-j}(r)_{k-1}}{\binom{k}{j}}
 \le    \frac {2^{k-j+2}  k^2 r^{2k-j-1} \delta}{\binom{k}{j}n^k},
\end{align*}
hence
\begin{align*}
& \left| \sum_{(J \in \cliques k {r+k}(G) : K \subseteq J)} \sum_{e \in E(J)} \frac {\omega_K^e(\kappa-\extensions e r)w} {\extensions e {r+k}} \right| 
\leq \frac {n^k} {k!} \sum_{j=0}^k \binom k {k-j} \binom{r}{j}\frac {2^{k-j+2}  k^2 r^{2k-j-1}\delta}{\binom{k}{j}n^k} \cdot w\\
& \qquad \qquad \leq \frac {2^{k+2} k^2 r^{2k-1} \delta w} {k!}\sum_{j=0}^k \frac {1}{2^{j}j!}				\leq  w,
\end{align*}
as required.
\end{proof}

In some cases it is possible to sharpen the computations in the proof of Theorem~\ref{hypergraphs} to lower the minimum codegree that guarantees the existence of a fractional $\clique k r$-decomposition.
Of particular interest is the case where $r=k+1$.
In this case, equation~\eqref{eq:irow} can be solved exactly to obtain $\alpha_j = (-1)^{k-j}/(k+1)\binom k j$.%
\COMMENT{$\binom {k-i} 0 \binom {1+i} i \alpha_i + \binom {k-i} 1 \binom {1+i} {1+i} \alpha_{i+1} = 0$, so $\alpha_i = - \frac {k-i} {i+1} \alpha_{i+1}$ with $\alpha_k = 1/\binom {k+1} k = 1/(k+1)$.}
Redoing the computation above with these correct values for $\omega^e_K$ shows that a minimum codegree of $(1- 1/k^2(k+1)2^{2k+1})n$ already guarantees a fractional $\clique k {k+1}$-decomposition, a substantial improvement over substituting $r=k+1$ into our general result.%
\COMMENT{$\left| \frac {\omega_K^e(\kappa-\extensions e r)} {\extensions e {r+k}} \right|  \leq \frac {\frac 1 {(k+1) \binom k i} \cdot 2k\delta n \binom {k+1} 2} {n^{k+1} / (k+1)!} = \frac {2\delta k! k \binom {k+1} 2} {n^k \binom k i}$.
The sum is then at most $\frac {n^k} {k!} \sum_{j=0}^k \binom k {k-j} \binom {k+1} j \frac {2\delta k! k \binom {k+1} 2} {n^k \binom k j} = \delta k^2 (k+1) \sum_{j=0}^k \binom {k+1} j$.

Also need $\delta \leq 1/k2^{2k+3}$ for \eqref{bottom}.}

% end hypergraphs

% start graphs

\section{Bounds on the number of cliques} \label{sec:num-cliques}

We now turn to the special case of graphs. 
As for the more general case of hypergraphs, we shall be interested in the number $\kappa_e^{(r)}$ of $r$-cliques containing an edge $e$. 
In this section we first prove the following proposition relating the number of cliques of different sizes in a graph with high minimum degree, which is used repeatedly throughout the paper.

\begin{proposition}\label{cliqnos}
Let $r,n\in \N$, $\d:=1/2r$, and let~$G$ be a graph on~$n$ vertices with $\d(G)\geq (1-\d)n$. Then, for each $i\in[r]$,%
\[
%(r/n)^ik_r\leq 
k_{r-i}\leq (2r/n)^ik_r.
\] 
\end{proposition}
\begin{proof}
Let $i\in[r]$. For each clique $K\in \KK_{r-i}$, using the minimum degree of~$G$, the number of cliques in $\KK_r$ containing $K$ is at least
\[
\frac{1}{i!}\prod_{j=1}^i(n-(r-i+j-1)\d n)\geq \frac{1}{i!}\left(\frac n2\right)^i.
\]
Each clique $K\in \KK_r$ contains $\binom{r}{i}$ cliques in $\KK_{r-i}$. Therefore,
\[
\frac{1}{i!}\left(\frac n2\right)^ik_{r-i}\leq \binom{r}{i}k_r\leq \frac{r^i}{i!}k_r,
\]
and thus $k_{r-i}\leq (2r/n)^ik_r$.
%
%Furthermore, each clique in $\KK_{r-i}$ is in at most $\binom{n}{i}$ cliques in $\KK_r$. Therefore, $k_{r-i}\geq \binom{r}{i}k_r/\binom{n}{i}\geq (r/n)^ik_r$.
\end{proof}

Our next lemma gives a range of bounds on the number of cliques containing a fixed smaller clique (and, in particular then, an edge).

\begin{lemma}\label{cliqest} Let $r,n\in \N$ and $\d\leq 1/2r$, and let $G$ be a graph with $n$ vertices and $\d(G)\geq (1-\d) n$. 
Then, for each integer $t < r$ and each subset $Z\subset V(G)$, with $|Z|=t$ and $G[Z]\in \KK_t$, we have
\begin{enumerate}[label = \rm{(\roman{enumi})}]
\item $| \kappa_{Z}^{(r)} - k_{r-t} | \leq 2t\delta r k_{r-t}$, and \label{est1}
\item $\big|\kappa_Z^{(r)}-k_{r-t}+|\bigcup_{z\in Z}N^c(z)|k_{r-t-1}\big|\leq 6(t\delta r)^2k_{r-t}$. \label{est2}
\item For each $xy\in E(G)$, we have \label{est3}
\[
\Big|\kappa^{(r)}_{xy}-k_{r-2}-\sum_{i=1}^3(-1)^i \sum_{Y\subset N^c(x)\cup N^c(y):|Y|=i} \kappa_Y^{(r-2)}\Big|
\leq 11(\d r)^4k_{r-2}.
\]
\end{enumerate}
\end{lemma}
\begin{proof}
Given $Z\subset V(G)$ with $|Z|=t$ and $G[Z]\in \KK_t$, we can obtain an $r$-clique $K$ containing $Z$ by extending $Z$ by the vertex set of an $(r-t)$-clique which lies in $\bigcap_{z\in Z}N(z)$.
By the inclusion-exclusion principle,
\begin{align*}
\kappa_{Z}^{(r)}
& = k_{r-t} - \big|\big\{K \in \KK_{r-t} : V(K) \cap \bigcup_{z\in Z}N^c(z) \neq \emptyset \big\}\big|  \nonumber 
\\
 & = k_{r-t} + \sum_{i=1}^{r-t} (-1)^{i}\sum_{Y\subseteq \bigcup_{z\in Z}N^c(z) \colon |Y|=i}\kappa_{Y}^{(r-t)}. 
\end{align*}
So by the Bonferroni inequalities (which say that the natural sequence of partial sums taken from the inclusion-exclusion formula alternately over- and underestimate the size of a union of sets),%
\COMMENT{Successive approximations to inclusion-exclusion alternately over- and underestimate.}
for each $\ell \leq r-t+1$,
\begin{multline}
\Big|\kappa_{Z}^{(r)}- k_{r-t} - \sum_{i=1}^{\ell-1} (-1)^{i}\sum_{Y\subseteq \bigcup_{z\in Z}N^c(z) \colon  |Y|=i}\kappa_{Y}^{(r-t)}\Big| \\
 \leq \sum_{Y\subseteq \bigcup_{z\in Z}N^c(z) \colon |Y|=\ell}\kappa_{Y}^{(r-t)}
 \leq \binom {t\delta n} \ell k_{r-t-\ell} \leq \frac {(t \delta n)^\ell} {\ell!} k_{r-t-\ell} \leq \frac {(2 t \delta r)^\ell} {\ell!} k_{r-t}, \label{inclusion-exclusion}
\end{multline}
where we have used Proposition~\ref{cliqnos} in the final inequality.
As $\ell$ increases, we obtain an increasingly accurate estimate for $\kappa_Z^{(r)}$ (provided $t$ is not too large). 
In particular, setting $\ell=1$ we gain \ref{est1}, and setting $\ell=4$ in the case where $|Z|=2$ we gain (iii).

Finally, for \ref{est2}, using \ref{est1} with clique size $r-t$ and set $\{x\}$ for each $x\in \bigcup_{z\in Z}N^c(z)$, we have
\begin{align*}
\Big|\kappa_Z^{(r)}-k_{r-t}+\big|\bigcup_{z\in Z}N^c(z)\big|k_{r-t-1}\Big|
 & = \Big|\kappa_Z^{(r)}-k_{r-t}+\sum_{x\in \bigcup_{z\in Z}N^c(z)}k_{r-t-1}\Big|
\\ &
\leq \Big|\kappa_Z^{(r)}-k_{r-t}+\sum_{x\in \bigcup_{z\in Z}N^c(z)}\kappa_{\{x\}}^{(r-t)}\Big|
\\
&\qquad +\sum_{x\in \bigcup_{z\in Z}N^c(z)}|\kappa_{\{x\}}^{(r-t)}-k_{r-t-1}|
\\
&\hspace{-0.35cm}\stackrel{ \eqref{inclusion-exclusion},\ref{est1}}{\leq} 2(t\d r)^2k_{r-t}+\Big|\bigcup_{z\in Z}N^c(z)\Big|2\d rk_{r-t-1}
\\
%&\leq 2(t\d r)^2k_{r-t}+\d t n (2\d t r k_{r-t-1})
%\\
&\leq 2t^2\d^2r^2k_{r-t}+\d t n\cdot 4\d r^2k_{r-t}/n\leq 6t^2\d^2 r^2 k_{r-t},
\end{align*}
where we have used Proposition~\ref{cliqnos} in the penultimate inequality.
\end{proof}

As noted in Section~\ref{sketch}, we shall want to construct edge-gadgets using only some of the $r$-cliques in the graph $G$. 
We will in fact have a small subset $X\subset V(G)$, and wish to avoid using $r$-cliques which have a large intersection with $X$. 
Our final result of this section demonstrates that there are not many such cliques.
\begin{proposition}\label{KXsmall} 
Let $r\geq 3$,%
\COMMENT{Need 3 as apply for $r-1$ with $r \geq 4$.}
$n\in \N$ and $\d:=1/600r^{3/2}$.
Let~$G$ be a graph on~$n$ vertices with $\d(G)\geq (1-\d)n$, and let $X\subset V(G)$ with $|X| \leq \delta r n$. 
Let
\[
\AA:=\{K\in \KK_r:|V(K)\cap X| \geq r^{1/2}\}.
\]
Then $|\AA|\leq k_r/r^2$.
\end{proposition}

\begin{proof}
Let $t := \lceil r^{1/2} \rceil$.
%, so that $(200/e)^t \ge (\sqrt2 t)^6 \ge r^3$.%
%\COMMENT{A: Note that $2^t\geq \sqrt{2} t$. (Let $f(t) := 2^t -  \sqrt{2}t$. 
%So $df/dt = 2^t \ln 2- \sqrt{2} \ge 0$ for $t \ge 2$. Thus we have $f(t) \ge 0$ for $t \ge 2$) Thus $64^t = 2^{6t} \ge (\sqrt{2}t)^6$.}
Using Proposition~\ref{cliqnos}, we have that
\begin{align*}
|\AA| & \leq \sum_{i=t}^r \binom {\delta r n}      i   k_{r-i} 
        \leq \sum_{i=t}^r \frac {(\delta r n)^i}  {i!} k_{r-i} 
        \leq \sum_{i=t}^r \frac {(2\delta r^2)^i} {i!} k_r 
				   = \sum_{i=t}^r \frac {(r^{1/2}/300)^i} {i!} k_r \\
		  & \leq r \frac {(r^{1/2}/300)^t} {t!} k_r 
			  \leq r (t/300)^t (e/t)^t k_r 
				   \le r k_r / 100^t
				%\leq r k_r / t^6 
				\leq k_r / r^2. \qedhere
\end{align*}
\end{proof}

\section{Adding weight over an edge}\label{secweight}

Recall from Section~\ref{sketch} that, in order to turn our initial uniform weighting into a fractional clique decomposition, our aim is to construct edge-gadgets which adjust the weight over an edge~$e$ by adjusting the weights of some $r$-cliques. 
In our proof of Theorem~\ref{hypergraphs} (for $k =2$), we implicitly used an edge-gadget $\psi_e$ that was the average of a basic edge-gadget $\psi_e^{J}$ over all cliques $J \in \KK_{r+2}$ containing~$e$ (defined more explicitly in Section~\ref{sketch}). 
This averaging ensured that the weight of any given clique was not altered so much that it became negative. 
Using some simple preprocessing (namely removing $r$-cliques one-by-one until any further removal violates the minimum degree condition) we can reduce the total adjustment we need to make to the initial weighting. 
In fact, for most of the edges we will only need to make small adjustments, leaving us most concerned with certain `bad' edges. 
Moreover, for each edge $e\in E(G)$, the edge-gadget $\psi_e$ requires larger adjustments to be made to those cliques whose intersection with $V(e)$ is larger. 
Thus we are limited by the adjustment we ask from cliques which contain many vertices in bad edges. 
By avoiding basic edge-gadgets which require adjustment to the weight of such cliques, we can reduce the minimum degree condition needed for these techniques to work.

%If a clique $K \in \KK_r$ is contained in many $(r+2)$-cliques (i.e. $\kappa_{V(K)}^{(r+2)}$ is large), then $\psi_e^{K^+} (K) \ne 0$ for many $e$ and $K^+ \in \KK_{r+2}$ containing~$e$.
%So there is the danger that when adjusting the weights of all the edges of~$G$, we will change the weight of~$K$ too much (i.e. the final weight of $K$ will not lie in $[0,1]$).
%Thus our aim is to avoid such cliques $K$ when adjusting the weights of the edges of~$G$.
In this section, we give sufficient conditions on a subset $\mathcal{A} \subseteq \KK_r$ to ensure we can construct good edge-gadgets that only change the weights of cliques in $\mathcal{A}$. 
%Essentially, we say $\mathcal{A}\subset \KK_r$ is \emph{well-distributed} if for each edge we can construct many basic edge-type gadgets for that edge using only cliques in $\AA$, before averaging over such basic edge-type gadgets to achieve Lemma~\ref{addweightedge}. 

\begin{definition}\label{edgedist}
Given a graph~$G$ we say that $\AA\subset \KK_r$ is \emph{well-distributed} if, for each $e\in E(G)$, there are at least $k_r/2$ sets $A\subset V(G) \setminus V(e)$ for which $|A|=r$ and, for each subset $B\subset V(e)\cup A$ with $|B|=r$, $G[B]\in \AA$.
\end{definition}

Informally, $\AA$ is well-distributed if the $r$-cliques it contains can be used to build many different basic edge-gadgets $\psi_e$ for each edge $e$.

\begin{lemma}\label{addweightedge} 
Let $r\geq 3$ and let $G$ be a graph on $n$ vertices.
Suppose that $k_r>0$ and that $\AA\subset \KK_r$ is well-distributed. Then for each edge $e\in E(G)$ there exists a function $\psi_e:\AA\to \R$ so that the following holds.

\begin{enumerate}[label = \rm{(\roman{enumi})}]
\item For all $e,f\in E(G)$,
\[
\sum_{K\in \AA\colon f\in E(K)}\weights{\psi}{e}{K}= \mathbf{1}_{\{e=f\}}.
\]
\item For all $K\in \AA$ and $e\in E(G)$, if $i=|V(K)\cap V(e)|$, then
$
|\weights{\psi}{e}{K}|\leq 6n^{i}/r^{i}k_r$.
\end{enumerate}
\end{lemma}
\begin{proof} 
The proof idea is similar to that of Proposition~\ref{gadget}.
For each edge $e\in E(G)$, let $\HH_e$ be the set of sets $A\subset V(G)\setminus V(e)$ for which $|A|=r$ and, for each subset $B\subset V(e)\cup A$ with $|B|=r$, $G[B]\in \AA$.
As $\AA$ is well-distributed, $|\HH_e|\geq k_r/2$. For each clique $K\in \AA$, let $\a_{e,K}$ be the number of sets $A\in \HH_e$ for which $K \in \KK_r[A\cup V(e)]$. 
For each edge $e\in E(G)$ and clique $K\in \AA$, let
\[
\weights{\phi}{e}{K}:=\begin{cases}
\frac{2}{r(r-1)} & \text{ if }|V(K)\cap V(e)|=2,\\
-\frac{r-2}{r(r-1)} & \text{ if }|V(K)\cap V(e)|=1,\\
\frac{r-2}{r} & \text{ if }|V(K)\cap V(e)|=0,
\end{cases}
\]
and let $\weights{\psi}{e}{K}:=\a_{e,K}\weights{\phi}{e}{K}/|\HH_e|$. We will now show that $\psi_e$ satisfies the requirements of the lemma.\
\COMMENT{Note that this is just the solution of \eqref{eq:irow}.}

Firstly, let $e,f\in E(G)$, and $A\in\HH_e$ with $V(f)\subset A\cup V(e)$. If $f=e$, then, as $|A|=r$, there are $\binom{r}{2}$ $r$-cliques $K\in \KK_r[A\cup V(e)]$ with $f=e\in E(K)$. Thus
\[
\sum_{K \in \KK_r [A\cup V(e)]\colon f\in E(K)}\weights{\phi}{e}{K}=1.
\]
If $f$ and $e$ share precisely one vertex, then for each $i\in \{1,2\}$ there are $\binom{r-1}{i}$ $r$-cliques $K\in \KK_r[A\cup V(e)]$ with $f\in E(K)$ and $|V(K)\cap V(e)|=i$.  Thus
\[
\sum_{K \in \KK_r [A\cup V(e)]\colon f\in E(K)}\weights{\phi}{e}{K}=\binom{r-1}{2}\frac{2}{r(r-1)}-(r-1)\frac{r-2}{r(r-1)}=0.
\]
If $f$ and $e$ share no vertices, then for each $i\in \{0,1,2\}$ there are $\binom{2}{2-i}\binom{r-2}{i}$ cliques $K\in \KK_r[A\cup V(e)]$ with $f\in E(K)$ and $|V(K)\cap V(e)|=i$. Thus
\[
\sum_{K \in \KK_r[A\cup V(e)]\colon f\in E(K)}\weights{\phi}{e}{K}=\binom{r-2}{2}\frac{2}{r(r-1)}-2(r-2)\frac{r-2}{r(r-1)}+\frac{r-2}{r}=0.
\]
Therefore,
\begin{align*}
\sum_{K\in \AA\colon f\in E(K)}\weights{\psi}{e}{K} &=  \sum_{K\in \AA:f\in E(K)}  \frac{1}{|\HH_e|} \sum_{A\in\HH_e : K \in \KK_r[A\cup V(e)] }\weights{\phi}{e}{K}
\\
&=\frac{1}{|\HH_e|}\sum_{A\in \HH_e}\sum_{ K \in \KK_r[A\cup V(e)]:f\in E(K)}\weights{\phi}{e}{K}
=\frac{1}{|\HH_e|}\sum_{A\in \HH_e}\mathbf{1}_{\{e=f\}}=\mathbf{1}_{\{e=f\}},
\end{align*}
as required.

Secondly, fix an edge $e\in E(G)$ and a clique $K\in \AA$, and let $i:=|V(K)\cap V(e)|$. There are at most $\binom{n}{i}$ sets $A\in \HH_e$ for which $K \in \KK_r[A\cup V(e)]$, and thus $\a_{e,K}\leq n^{i}$. As mentioned previously, we have $|\HH_e|\geq k_r/2$, and we can observe that $|\weights{\phi}{e}{K}|\leq 3/r^i$. Therefore,% as required,
\[
|\weights{\psi}{e}{K}|\leq (2n^{i}/k_r)|\weights{\phi}{e}{K}|\leq 6n^{i}/r^{i}k_r.\qedhere
\]
\end{proof}

We will initially weight each clique with $1/\kappa$, where $\kappa:=k_{r-2}-2\delta nk_{r-3}$. 
As we will see later (in Lemma~\ref{breakdown}), this gives an almost fractional $K_r$-decomposition of $G$. 
Let $\pi:E(G)\to\R$ record the amount of weight we wish to add over each edge to achieve a fractional $K_r$-decomposition. 
We wish to know whether we can make these adjustments using edge-gadgets while keeping the weights on the $r$-cliques positive.
%Suppose that $\pi(e)$ is the amount of weight we would like to add to the edge~$e$. 
%Note that the case when $|\pi(e)|$ is large (roughly) corresponds to the case when $\kappa_e^{(r)}$ is much larger/smaller compare to the average $\sum_{f \in E(G)} \kappa_f^{(r)} /e(G)$.
The next lemma, Lemma~\ref{smoothweight}, says that we can make these adjustments while changing the weight of each clique by no more than $1/2\kappa$, provided that the adjustments given by $\pi$ are on average quite small and $\pi$ is sufficiently `smooth'. 
That is, $|\pi|$ is not significantly above average for any edge, and the average of $|\pi|$ around each vertex is even more restricted.
Before we state Lemma~\ref{smoothweight}, we formalise these properties by the following definition.

\begin{definition}\label{smoothdefn} Given a graph~$G$ and $r \in \mathbb N$, a function $\pi:E(G)\to\R$ is \emph{$r$-smooth} if
\begin{enumerate}[label= \rm{(A\arabic{enumi})}]
\item for each edge $xy\in E(G)$, $|\pi(xy)|\leq 1/10^4$,\label{good1}
\item for each vertex $x\in V(G)$, $\sum_{y\in N(x)}|\pi(xy)|\leq n/10^4r$, and \label{good2}
\item $\sum_{xy\in E(G)}|\pi(xy)|\leq n^2/10^4r^2$.\label{good3}
\end{enumerate}
\end{definition}

Note that (A1) does not imply (A2), and (A2) does not imply (A3). 

The intuition behind the definition of smoothness is as follows.
%BEN OLD PARTIAL: The basic edge-gadget $\phi_e$ increases the weight over $e$ by first increasing the weight of some $r$-cliques containing $e$, then makes further adjustments to cancel out the change in weight over every other edge of these cliques.
To construct a basic edge-gadget $\phi_e$, we increased only the weight over $e$ by first increasing the weight of some $r$-cliques containing $e$, then making further adjustments to cancel out the change in weight over every other edge of these cliques.
These cancellations introduce an inherent inefficiency and mean that we can only hope to correct errors of average size $O(1/r^2)$ (cf.\ (A3)), although we can handle slightly larger localised errors (cf.\ (A1) and (A2)).
%BEN OLD: This introduces

%RICHARD OLD: We now show that if $\pi$ is $r$-smooth, then there is a function $\omega: \KK_r \to \R$ such that if we make these adjustments to the weight of the cliques in $\KK_r$ the change in the weight over each edge $e$ is $\pi(e)$ and $|\omega(K)|\leq 1/2\kappa$ for each $ K \in \KK_r$. We can only make the required changes using edge-gadgets if on average each edge needs altered by $O(1/r^2)$. Essentially, in a basic edge-gadget for the edge $e$ we add and remove $r$-cliques until the adjustments cancel out on every edge except for $e$. For each clique $K$ containing $e$ that we add weight to, the weight on the other edges in $K$ is cancelled out. This introduces an inefficiency and as a result we can only move $O(1/r^2)$ of the weight over the edges around.

\begin{lemma}\label{smoothweight}  Let $r\geq 4$, $n\in \N$ and $0\leq\d\leq 1/24r$. Let~$G$ be a graph with~$n$ vertices and $\d(G)\geq (1-\d)n$.
Let $\kappa:=k_{r-2}-2\d n k_{r-3}$, and let $\pi:E(G)\to\R$ be $r$-smooth. 
Then there exists a function $\omega : \KK_r \to \R$ so that $|\omega(K)|\leq 1/2\kappa$ for all $K \in \KK _r$ and, for each $e\in E(G)$, 
\[
\sum_{K\in \KK_r\colon e\in E(K)} \omega(K) = \pi(e).
\]
\end{lemma}

\begin{proof}%[Proof of Lemma~\ref{smoothweight}]
Let $\gamma:=1/10^4r^2$, and let
\begin{align}
\label{AAdefn}
\AA:= \Big\{ K\in \KK_r:\sum_{e'\in E(K)}| \pi(e')| \leq 72r^2\gamma\;\text{ and }\sum_{e' \in E(G) :|V(K)\cap V(e')|\geq 1} |\pi(e')|\leq 48rn\gamma\Big\}.
\end{align}
We will show that $\AA$ is well-distributed, and then define $\omega$ using the edge-gadgets $\psi_e$ obtained by applying Lemma~\ref{addweightedge} with~$\AA$.

Note that, using Proposition~\ref{cliqnos}, 
\begin{equation}\label{kappaest}
k_{r-2}\geq \kappa\geq k_{r-2}-4\d r k_{r-2}\geq 5k_{r-2}/6 > 0.
\end{equation}
For each $e\in E(G)$, let 
\begin{equation}\label{Hedefn}
\HH_{e}:=\{A\subset V(G)\setminus V(e):|A|=r\;\text { and }\;G[A\cup V(e)]\in \KK_{r+2}\},
\end{equation}
and note that, using Lemma~\ref{cliqest}\ref{est1}, we have\COMMENT{Here is the place we use $\delta \le 1/24r$}
\begin{equation}\label{Hebound}
|\HH_e|=\kappa_{V(e)}^{(r+2)}\geq k_r-4\d(r+2)k_r\geq 3k_r/4.
\end{equation}
Let 
\[
\HH_{e,1}:=\Big\{A\in \HH_e:\sum_{e'\in E(G[A\cup V(e)])}|\pi(e')|\leq 72r^2 \gamma\Big\}.
\]
\begin{claim}\label{He1bound} For each $e\in E(G)$, $|\HH_e\setminus \HH_{e,1}|\leq k_r/8$.
\end{claim}
\begin{proof}[Proof of Claim~\ref{He1bound}]
For each $i\in \{0,1,2\}$, each edge $e'\in E(G)$ with $|V(e')\cap V(e)|=i$ is in at most $k_{r+i-2}$ of the graphs $G[A\cup V(e)]$, with $A\in \HH_e$. We therefore have that, using Proposition~\ref{cliqnos} and \ref{good1}--\ref{good3},
\begin{align*}
|\HH_{e}\setminus\HH_{e,1} |(72r^2\gamma)&\leq \sum_{A\in \HH_e}\sum_{e'\in E(G[A\cup V(e)])}|\pi(e')|
\leq \sum_{i=0}^2\sum_{e':|V(e)\cap V(e')|=i}k_{r+i-2}|\pi(e')|\\
&\leq n^2\gamma k_{r-2}+2rn\gamma k_{r-1}+r^2\gamma k_r
\leq 9r^2\gamma k_r,
\end{align*}
hence $|\HH_e\setminus \HH_{e,1}|\leq k_r/8$.
\end{proof}

Now let 
\[
\HH_{e,2}:=\Big\{A\in \HH_e:\sum_{e'\in E(G):|V(e')\cap (A\cup V(e))|\geq 1}|\pi(e')|\leq 48r n \gamma\Big\}.
\]
\begin{claim}\label{He2bound} For each $e\in E(G)$, $|\HH_e\setminus \HH_{e,2}|\leq k_r/8$.
\end{claim}
\begin{proof}[Proof of Claim~\ref{He2bound}]
Let $e'\in E(G)$.
If $V(e)\cap V(e')=\emptyset$, then there are at most $2k_{r-1}$ sets $A\in \HH_e$ for which $|V(e')\cap (A\cup V(e))|\geq 1$. 
If $|V(e)\cap V(e')|\geq 1$, then $|V(e')\cap (A\cup V(e))|\geq 1$ for every $A\in \HH_e$, and $|\HH_e|\leq k_r$.
We therefore have that, using Proposition~\ref{cliqnos}, \ref{good2} and \ref{good3},
\begin{align*}
|\HH_{e}\setminus\HH_{e,2} |(48rn\gamma)&\leq \sum_{A\in \HH_{e}}\sum_{e':|V(e')\cap (A\cup V(e))|\geq 1}|\pi(e')|\\
&\leq \sum_{e' \in E(G):|V(e)\cap V(e')|=0}2k_{r-1}|\pi(e')|+\sum_{e'\in E(G):|V(e)\cap V(e')|\geq 1}k_{r}|\pi(e')|
\\
&\leq 2n^2\gamma k_{r-1}+2rn\gamma k_r
\leq 6rn\gamma k_r,
\end{align*}
hence $|\HH_{e}\setminus\HH_{e,2}|\leq k_r/8$.
\end{proof}

For each $e\in E(G)$, let $\bar{\HH}_e:=\HH_{e,1}\cap\HH_{e,2}$, so that 
\begin{comment}
For each $e=xy\in E(G)$, using Proposition~\ref{cliqnos} we have that
\begin{align*}
|\HH_e|&=|\{K\in \KK_r:V(K)\subset N(x)\cap N(y)\}|\\
&=k_r-|\{K\in \KK_r:V(K)\cap (N^c(x)\cup N^c(y))\neq\emptyset\}|\\
&\geq k_r-2\d nk_{r-1}\geq (1-4\d r)k_{r}\geq 3k_r/4.
\end{align*}
\end{comment}
by~\eqref{Hebound} and Claims~\ref{He1bound} and~\ref{He2bound}, we have $|\bar{\HH}_e|\geq 3k_r /4-k_r/4\geq k_r/2$.
We can now check that the set $\AA$ defined by~\eqref{AAdefn} is well-distributed.

For each $A\in \bar{\HH}_e$ and every $r$-clique $K \in \KK_r[A\cup V(e)]$ we have from the definition of $\HH_{e,1}$ and $\HH_{e,2}$ that $K\in \AA$. 
Since $|\bar\HH_e|\geq k_r/2$ for each edge $e\in V(G)$, this implies that $\AA$ is well-distributed.
Thus by Lemma~\ref{addweightedge}, for each $e\in E(G)$, there exists a function $\psi_e:\AA\to \R$ so that the following holds.
\begin{enumerate}[label = (\alph{enumi})]
\item If $e'\in E(G)$, then $\sum_{K\in \AA\colon e'\in E(K)}\weights{\psi}{e}{K}=\mathbf{1}_{\{e'=e\}}$. \label{aaa}
\item For each $K\in \AA$, if $i=|V(K)\cap V(e)|$, then $|\weights{\psi}{e}{K}|\leq 6n^{i}/r^{i}k_r$. \label{bbb}
\end{enumerate}
Now, for each $K\in \AA$, let
\begin{equation}\label{wKdefn}
\omega(K):=\sum_{e\in E(G)}\weights{\psi}{e}{K}\pi(e),
\end{equation}
and for each $K\in \KK_r\setminus \AA$, let $\omega(K):=0$. Then, for each $e\in E(G)$,
\[
\sum_{K\in \KK_r:e\in E(K)}\omega(K) = \sum_{e'\in E(G)}\sum_{K\in \AA:e\in E(K)}\weights{\psi}{e'}{K} \pi(e') \stackrel{\rm{(a)}}{=} \sum_{e'\in E(G)}\mathbf{1}_{\{e'=e\}}\pi(e')=\pi(e),
\]
as required.
Moreover, \ref{bbb}, (\ref{wKdefn}), (\ref{AAdefn}), \ref{good3},~\eqref{kappaest} and Proposition~\ref{cliqnos} together imply that, for each $K\in \AA$,
\begin{align*}
|\omega(K)|&\leq \sum_{e\in E(K)}6|\pi(e)|n^2/r^2k_r+\sum_{e\in E(G):|V(K)\cap V(e)|=1} 6|\pi(e)|n/rk_r+\sum_{e\in E(G)}6|\pi(e)|/k_r
\\
&\leq 6(72r^2 \gamma)n^2/r^2k_r+6(48rn \gamma)n/rk_r+6(n^2\gamma)/k_r
\\
&\leq 1000n^2\gamma/k_r=n^2/10r^2k_r\leq 2/5k_{r-2}\leq 1/2\kappa. \qedhere
\end{align*}
\end{proof}

%----------------------------------- r^2 -----------------------------------------

\section{Fractional $K_r$-decompositions when $\d(G)\geq (1-1/10^5r^2)n$.} \label{escape-route}

The aim of this section is to prove Theorem~\ref{fracdecomp} under the stronger assumption that $\delta(G) \ge (1-\delta)n$ with $\delta:=1/10^5r^2$ (see Theorem~\ref{r2lemma} below).
We include a proof of this intermediate bound as it follows easily from Lemma~\ref{smoothweight}, and shows how we will make use of that lemma. 

As noted in Section~\ref{secweight}, after initially weighting the $r$-cliques uniformly with value $1/\kappa$ (where $\kappa:=k_{r-2}-2\delta nk_{r-3}$), Lemma~\ref{smoothweight} permits us to move a $\Omega(1/r^2)$ proportion of the weight over the edges around (subject to certain constraints) without making any of the $r$-clique weights negative. 
We will see, using Lemma~\ref{cliqest}, that we only need to adjust a $O(\d r)$ proportion of the weight over each edge to turn our initial uniform weighting into a fractional $K_r$-decomposition. 
Thus in the case when $\delta = O(1/r^3)$ is suitably small we can apply Lemma~\ref{smoothweight} to make this adjustment.
This corresponds to the argument presented in Section~\ref{sec:hyper}.

However, if we carry out some initial preprocessing (removing $r$-cliques until the minimum degree condition would be violated by any further removal) we can reduce the overall proportion of weight over the edges that we might need to move to $O(\d^2r^2)$. 
This allows us to use Lemma~\ref{smoothweight} even in the case when $\delta = O(1/r^2)$ is sufficiently small.

\begin{theorem}\label{r2lemma}
Let $r\geq 4$ and let $G$ be a graph with $n\geq 10^6 r^4$ vertices and $\d(G)\geq (1-1/10^5r^{2})n$. Then~$G$ has a fractional $K_r$-decomposition.
\end{theorem}

%In order to prove Theorem~\ref{r2lemma}, we start with the uniform weighting, i.e. we let $\kappa:=k_{r-2}-2\d n k_{r-3}$ and give each clique $K \in \KK_r$ weight $\kappa^{-1}$.
%So each edge~$e$ has weight $\kappa^{(r)}_e /\kappa$.
%For each edge $e \in E(G)$, let $\pi(e):= \kappa^{(r)}_e-\kappa$.
%Note that if we remove $\pi(e)/\kappa$ weight from each edge~$e$ then we would obtain a fractional $K_r$-decomposition. 
%Thus, by Lemma~\ref{smoothweight}, it suffices to show that $\pi/\kappa$ is smooth. 

\begin{proof}%[Proof of Theorem~\ref{r2lemma}]
Let $\d:=1/10^5r^{2}$. 
We may assume that we cannot remove any $r$-cliques from~$G$ while maintaining minimum degree at least $(1-\d)n$. 
Indeed, by removing a sequence of $r$-cliques from~$G$ we can find a subgraph~$H$ for which $\d(H)\geq (1-\d)n$ but for which removing any $r$-clique violates this minimum degree condition; if~$H$ has a fractional $K_r$-decomposition, then clearly~$G$ does also. 
Therefore, writing $X:=\{x\in V(G):d(x)\geq (1-\d)n+r-1\}$, we may assume that $G[X]$ is $K_r$-free. 
As each $v \in X$ has at most $\delta n$ non-neighbours in~$G$,
$\delta(G[X]) \geq |X| - \delta n = (1 - \delta n / |X|) |X|$,
so by Tur\'an's theorem $\delta n / |X| \geq 1/(r-1)$, i.e. $|X|\leq \d (r-1) n$.

%For each edge $e\in E(G)$, let $z_e$ be the number of cliques in $\KK_r$ which contain~$e$. 
For each edge $e\in E(G)$, we have, by Lemma~\ref{cliqest}\ref{est2}, that
\begin{equation}\label{sunny}
\big|\kappa^{(r)}_{e}-k_{r-2}+|N^c(x)\cup N^c(y)|k_{r-3} \big|\leq 24(\d r)^2k_{r-2}.
\end{equation}
Let $\kappa:=k_{r-2}-2\d n k_{r-3}$, so that, by Proposition~\ref{cliqnos}, $ \kappa \geq (1-4\d r)k_{r-2}\geq 9k_{r-2}/10$.
For each $e\in E(G)$, let $\pi(e):=\kappa^{(r)}_e-\kappa$, so that, by \eqref{sunny}, we have
\begin{align}
|\pi(e)|& \leq (2\d n-|N^c(x)\cup N^c(y)|)k_{r-3}+24(\d r)^2k_{r-2}. \label{chips}
\end{align}
In particular, together with Proposition~\ref{cliqnos} this implies that
\[
|\pi(e)| \leq  2\d nk_{r-3}+24(\d r)^2k_{r-2}
\leq (4\d r+24(\d r)^2) k_{r-2}\leq 9k_{r-2}/10^5r\leq \kappa/10^4r.
\]
For each $x\in V(G)$, then, $\sum_{y\in N(x)}|\pi(xy)/\kappa|\leq n/10^4r$.
Furthermore, using (\ref{chips}), Proposition~\ref{cliqnos}, and our assumption that $n\geq 10^6r^4$,
\begin{align*}
2\sum_{e \in E(G)}&|\pi(e)| =  \sum_{x\in V(G)}\sum_{y\in N(x)}|\pi(xy)|\hspace{8.5cm}\ 
\end{align*}
\vspace{-0.35cm}
\begin{align*}
\text{\ \;\;\;\;\;\;\ }&\leq \sum_{x\in V(G)}\sum_{y\in N(x)}(2\d n-|N^c(x)\cup N^c(y)|)k_{r-3}+24n^2(\d r)^2k_{r-2}
\\
&\leq 2\sum_{x\in X}\sum_{y\in N(x)}2\d nk_{r-3}+\sum_{x\notin X}\sum_{y\in N(x)\setminus X}(2\d n-|N^c(x)\cup N^c(y)|)k_{r-3}+24\d^2r^2n^2k_{r-2}
\\
&\leq 4\d^2r n^3k_{r-3}+\sum_{x\notin X}\sum_{y\in N(x)\setminus X}(2r+|N^c(x)\cap N^c(y)|)k_{r-3}+24\d^2r^2n^2k_{r-2}
\\
&\leq 8\d^2r^2 n^2k_{r-2}+2rn^2k_{r-3}+\sum_{x\notin X}\sum_{z\in N^c(x)}|N^c(z)|k_{r-3}+24\d^2r^2n^2k_{r-2}
\\
&\leq 32\d^2r^2 n^2k_{r-2}+4r^2nk_{r-2}+\d^2 n^3k_{r-3}
%\\
%&\leq n^2k_{r-2}/10^5r^2+4r^2nk_{r-2}+\d^2 n^3k_{r-3}
\\
&\leq n^2k_{r-2}/10^5r^2+ n^2k_{r-2}/10^5r^2+2\d^2rn^2k_{r-2}
% 2n^2k_{r-2}/10^4r^2+\d^2r n^2k_{r-3}+3n^2k_{r-2}/10^4r^2
\leq 9n^2k_{r-2}/10^5r^2\leq n^2\kappa/10^4r^2.
\end{align*}
Therefore, the function $\pi/\kappa:E(G)\to\R$ is $r$-smooth. Thus Lemma~\ref{smoothweight} implies that there is a function $\omega':\KK_r \to \R$ so that, for each $e\in E(G)$, $\sum_{K\in \KK_r:e\in E(K)}\omega'(K) = \pi(e)/\kappa$, and, for each $K\in \KK_r$, $|\omega'(K)|\leq 1/2\kappa$.

Define $\omega:\KK_r \to \R$ by setting $\omega(K):=1/\kappa -\omega'(K)$ for each $K \in \KK_r$. 
Then, for each $e\in E(G)$,
\[
\sum_{K\in \KK_r:e\in E(K)}  \omega(K) = \frac{\kappa_e^{(r)}-\pi(e)}{\kappa}=1,
\]
and, for each $K\in \KK_r$, $ \omega(K) \geq 1/\kappa -1/2\kappa \geq 0$.
Therefore, $\omega$ is a fractional $K_r$-decomposition of~$G$.
\end{proof}

%----------------------- Adding weight around a vertex -------------------------

\section{Adding weight around a vertex} \label{sec:vertexgadget}

After our initial preprocessing of the graph $G$ and the initial weighting of the $r$-cliques with $1/\kappa$, where $\kappa:=k_{r-2}-2\delta nk_{r-3}$, we may need to add/subtract on average a $\Omega(\delta^2r^2)$ proportion of the weight over each edge. Our edge-gadgets can only add/subtract weight over each edge if it is on average $O(1/r^2)$. 
Thus the techniques in Section~\ref{escape-route} require $\delta=O(1/r^2)$.

In order to increase the size of $\delta$, in this section we introduce `vertex-gadgets', defined explicitly below, which in this set-up are capable of adding/subtracting $\Omega(1/r)$ of the weight over each edge. However, while this is more efficient than using edge-gadgets, the vertex-gadgets can only change the weight of every edge around some vertex simultaneously by the same amount.

\begin{comment}
For the rest of the paper, we assume that $\delta = 1/10^4 r^{3/2}$.
This means that the function $\pi$ defined in the proof of Theorem~\ref{r2lemma} is no longer smooth.
Furthermore, this is far from the truth. 
Indeed, consider the following example. 
Suppose that $G$ is obtained from a random $(1- \delta)n$-regular graph on $n-1$ vertices by joining every vertex to a new vertex~$x$. 
Note that $\kappa^{(r)}_{xy} > \kappa^{(r)}_{e}$ for all $y \in V(G) \setminus \{x\}$ and $e \in E(G)$ with $x \notin V(e)$.
Furthermore, $\kappa^{(r)}_{e} \approx \kappa^{(r)}_{f}$ for all $e,f \in E(G)$ with $x \notin V(e) \cup V(f)$.
So our aim in this section is to reduce the weight of edges $xy$ for all $y \in N(x)$ simultaneously, using a \emph{vertex-gadget}.
\end{comment}
For a vertex $x \in V(G)$, a \emph{vertex-gadget} is a function $\xi_x : \KK_r \to \R$  such that for each edge $e \in E(G)$,
\[
\sum_{K\in \KK_r \colon e\in E(K)}\weights{\xi}{x}{K}=\begin{cases} 1 & \text{if }x\in V(e),\\ 0 & \text{if }x\notin V(e). \end{cases}
\]

In the next lemma, we show that, for each vertex $x \in V(G)$, there exists a function $\phi_x : \KK_r \to \R$ such that
\begin{enumerate}[label = \rm{(\roman{enumi})}]
\item for each $e \in E(K)$ with $x \notin V(e)$, $\sum_{K\in \KK_r \colon e\in E(K)}\weights{\phi}{x}{K} = 0$, and
\item for each $y \in N(x)$, $\sum_{K\in \KK_r \colon xy\in E(K)}\weights{\phi}{x}{K}$ is close to $1$.
\end{enumerate}
Thus $\phi_x$ is almost a vertex-gadget.
We will then use edge-gadgets to make the requisite corrections to $\phi_x$ to obtain an actual vertex-gadget---see Lemma~\ref{addweightvertex2}. 
(Thus we actually define $\phi_x$ on a certain subset $\AA \subseteq \KK_r$ instead of $\KK_r$ so that we can make these adjustments efficiently.)

\begin{lemma}\label{addweightvertex} Let $r\geq 4$, $0 < \d \leq 1/600r^{3/2}$ and $n\geq 32r^3$. 
Let $G$ be a graph on $n$ vertices with $\d(G)\geq (1-\d)n$.
Let $X:=\{x \in V(G):d_G(x)\geq (1-\d)n+r-1\}$ and suppose that $|X|\leq \d (r-1) n$.
Let $\AA:=\{K\in \KK_r:|V(K)\cap X|\leq r^{1/2}+2\}$.
Then, for each vertex $x \in V(G)$, there exists a function $\phi_x: \AA \to \mathbb{R}$ for which the following holds, where, for each $y\in N(x)$, we let $\tau_{x,y}:=1-\sum_{K\in \AA\colon xy\in E(K)}\weights{\phi}{x}{K}$.
\begin{enumerate}[label= \rm(B\arabic{enumi})]
\item If $x\in V(G)$ and $e\in E(G)$ with $x\notin V(e)$, then $\sum_{K\in \AA\colon e\in E(K)}\weights{\phi}{x}{K}=0$.\label{bad1}

\item For all $x\in V(G)$ and $y\in N(x)$, $|\tau_{x,y}|\leq 1/r^{1/2}$.\label{bad2}

\item For each $x\in V(G)$, $\sum_{y\in N(x)}|\tau_{x,y}|\leq n/r$.\label{bad3}

\item For all $K\in \AA$ and $x\in V(G)$, if $i=|V(K)\cap \{x\}|$, then\label{bad4}
$
|\weights{\phi}{x}{K}|\leq 2 n^{i+1}/r^{i+1}k_r$.
\end{enumerate}
\end{lemma}

\begin{proof} %Note that $n\geq 10/\d^2$.
For each vertex $x\in V(G)$, let $\HH_x$ be the set of sets $A\subset V(G)\setminus \{x\}$ for which $|A|=r$, $G[A\cup \{x\}]\in\KK_{r+1}$ and $|A\cap X|\leq r^{1/2} + 1$. For each $x\in V(G)$ and $K\in \KK_{r}$, let $\a_{x,K}$ be the number of sets $A\in \HH_x$ for which $K\in \KK_r[A\cup \{x\}]$,
and let
\[
\weights{\psi}{x}{K}:=
\begin{cases}
\frac{1}{r-1} & \text{ if }x\in V(K),\\
-\frac{r-2}{r-1} & \text{ if }x\notin V(K).
\end{cases}
\]
For each $x\in V(G)$, let $w_x:=k_{r-1}-(n-d(x)+\d n)k_{r-2}$. Note that, by Proposition~\ref{cliqnos},
\begin{equation}
w_x\geq k_{r-1}-2\d n k_{r-2} \geq k_{r-1}-4\d r k_{r-1} \geq 7k_{r-1}/8.\label{WWW}
\end{equation}
For each $K\in \AA$, let $\weights{\phi}{x}{K}:=\a_{x,K}\weights{\psi}{x}{K}/w_x$. 
We will now show that $\phi_x$ satisfies the requirements of the lemma.

First, let $x\in V(G)$ and let $e\in E(G)$ with $x\notin V(e)$. 
If $A\in\HH_x$ and $V(e)\subset A\cup \{x\}$, then, for $i\in \{0,1\}$, there are $\binom{r-2}{i}$ cliques $K\in \KK_r$ with $K\in \KK_r[A\cup \{x\}]$, $e\in E(K)$ and $|V(K)\cap \{x\}|=i$. 
Thus,
\[
\sum_{K\in \KK_r [A\cup \{x\}]\colon e\in E(K)}\weights{\psi}{x}{K}=(r-2)\frac{1}{r-1}-\frac{r-2}{r-1}=0.
\]
Therefore if $x\in V(G)$ and $e\in E(G)$ with $x\notin V(e)$, we have
\begin{align*}
\sum_{K\in \AA\colon e\in E(K)}\weights{\phi}{x}{K}&=
\sum_{K\in \AA\colon e\in E(K)} \frac{1}{w_x} \sum_{A\in \HH_x\colon K\in \KK_r [A\cup\{x\}]}\weights{\psi}{x}{K}
\\
&=\frac{1}{w_x}\sum_{A\in \HH_x}\sum_{K \in \KK_r[A\cup \{x\}]\colon e\in E(K)}\weights{\psi}{x}{K}=0.
\end{align*}
(In the second equality we use that each $K \in \mathcal K_r[A \cup \{x\}]$ lies in $\mathcal{A}$ by the definition of $\mathcal H_x$.)
Therefore \ref{bad1} holds.

Now let $x\in V(G)$ and $y\in N(x)$. If $A\in\HH_x$ and $y\in A$, then there are $r-1$ cliques $K\in \KK_r$ with $K \in \KK_r [A\cup \{x\}]$ and $xy\in E(K)$.  
Thus, 
\[
\sum_{K \in \KK_r[A\cup \{x\}]\colon xy\in E(K)}\weights{\psi}{x}{K}=1.
\]
Let $w_{x,y}$ be the number of sets $A\subset V(G)$ for which $A\in\HH_x$ and $y\in A$.
Then
\begin{equation}
\sum_{K\in \AA\colon xy\in E(K)}\weights{\phi}{x}{K}=\frac{1}{w_x}\sum_{A\in \HH_x}\sum_{K\in  \KK_r[A\cup \{x\}]\colon xy\in E(K)}\weights{\psi}{x}{K}=\frac{w_{x,y}}{w_x}.\label{XXX}
\end{equation}
In the last equality we use that each $K \in \mathcal K_r[A \cup \{x\}]$ (with $xy \in E(K)$) lies in $\mathcal{A}$ by the definition of $\mathcal H_x$.

\begin{claim}\label{RRR1} For each $x\in V(G)$ and $y\in N(x)$,
\[
|w_{x,y}-w_x+(n-d(y)-\d n)k_{r-2}|\leq |N^c(x)\cap N^c(y)|k_{r-2}+(24\d^2 (r+1)^2+2/r^2)k_{r-1}.
\]
\end{claim}

\begin{proof}[Proof of Claim~\ref{RRR1}]
By Proposition~\ref{KXsmall}, there are at most $k_{r-1}/(r-1)^2\leq 2k_{r-1}/r^2$ cliques $K \in \KK_{r-1}$ for which $|X\cap V(K)|\geq r^{1/2}$.
Note that if $xy \in E(G)$, $K'$ is an $(r+1)$-clique containing~$xy$, and $| ( V(K') \setminus \{x,y\} ) \cap X| \le r^{1/2}$ then $V(K') \setminus \{x\}\in \mathcal{H}_x$.
Thus %for each $x\in V(G)$ and $y\in N(x)$,
\begin{align}\label{DD1}
\big|w_{x,y}-\kappa_{xy}^{(r+1)}\big|\leq 2k_{r-1}/r^2.
\end{align}
%There are at most $\binom{|N^c(x)\cup N^c(y)|}{2}k_{r-3}$ cliques $K\in \KK_{r-1}$ which have two or more vertices in $N^c(x)\cup N^c(y)$.
Then, by Lemma~\ref{cliqest}\ref{est2} and (\ref{DD1}), we have that
\begin{equation*}%\label{DD2}
\big| w_{x,y}-k_{r-1}+|N^c(x)\cup N^c(y)|k_{r-2} \big| \leq 24(\d (r+1))^2k_{r-1}+2k_{r-1}/r^2.
\end{equation*}
Thus,
\begin{align*}
|w_{x,y}&-w_x+(n-d(y)-\d n)k_{r-2}| 
=|w_{xy}-k_{r-1}+(2n-d(x)-d(y))k_{r-2}|
\\
&\leq \big|w_{x,y}-k_{r-1}+|N^c(x)\cup N^c(y)|k_{r-2}\big|+\big||N^c(x)\cup N^c(y)|-(2n-d(x)-d(y))\big|k_{r-2}
\\
&\leq (24\d^2 (r+1)^2+2/r^2)k_{r-1}+|N^c(x)\cap N^c(y)|k_{r-2}. \qedhere
\end{align*}
\end{proof}

By Claim~\ref{RRR1}, for each $x\in V(G)$ and $y\in N(x)$, using Proposition~\ref{cliqnos}, we have\COMMENT{Here, we use the fact that $\delta \ge 1/32r^{3/2}$}
\begin{align}
|w_{x,y}-w_x| &\leq \d n k_{r-2} +\d n k_{r-2} + (24 \d^2 (r+1)^2+2/r^2)k_{r-1} \nonumber\\
& \leq (4\d r + 24 \d^2 (r+1)^2+2/r^2) k_{r-1}\leq k_{r-1}/2r^{1/2}. \label{DD4}
\end{align}
For each $x\in V(G)$ and $y\in N(x)$, recall that
\begin{equation}\label{DD3}
\tau_{x,y}=1-\sum_{K\in \AA \colon  xy\in E(K)}\weights{\phi}{x}{K}\stackrel{(\ref{XXX})}{=}(w_{x}-w_{x,y})/w_x.
\end{equation}
Therefore (\ref{WWW}), (\ref{DD4}) and (\ref{DD3}) together imply that for each $x\in V(G)$ and $y\in N(x)$, $|\tau_{x,y}|\leq 1/r^{1/2}$, and thus \ref{bad2} holds.

By Claim~\ref{RRR1}, we have, for each $x\in V(G)$, that
\begin{align*}
\sum_{y\in N(x)}&|w_x-w_{x,y}| \\
& \leq \sum_{y\in V(G)} \big( |n-d(y)-\d n|+|N^c(x)\cap N^c(y)| \big) k_{r-2}+(24\d^2 (r+1)^2+2/r^2)nk_{r-1}\\
& \leq \Big(|X|\d n +rn+\sum_{z\in N^c(x)}|N^c(z)|\Big)k_{r-2}+(24\d^2 (r+1)^2+2/r^2)nk_{r-1}\\
& \leq (\d^2 (r-1) n^2+rn+\d^2n^2)k_{r-2}+(24\d^2 (r+1)^2+2/r^2)nk_{r-1}\\
& \stackrel{\mathclap{\text{P\ref{cliqnos}}}}{\leq} 
    \big( 2 \d^2 r^2+ 2r^2/n + 2r\d^2 + 24\d^2 (r+1)^2+2/r^2 \big)nk_{r-1} 
\leq 7 n k_{r-1} /8r,
\end{align*}
where the final inequality is due to the fact that $\delta \leq 1/600r^{3/2}$ and $n \ge 32r^3$.
Together with (\ref{WWW}) and (\ref{DD3}) this implies that, for each $x\in V(G)$, $\sum_{y\in N(x)}|\tau_{x,y}|=\sum_{y\in N(x)}|w_{xy}-w_x|/w_x\leq n/r$, which proves \ref{bad3}.

Finally, for each vertex $x\in V(G)$ and clique $K\in \KK_r$, setting $i: =|V(K)\cap \{x\}|$, there are at most $\binom{n}{i}$ sets $A\in \HH_x$ for which $K\in \KK_r [A\cup \{x\}]$, and thus $\a_{x,K}\leq n^{i}$. 
Moreover, $k_r \le n k_{r-1}/r$ and $|\psi_x(K)| \leq 4/3r^i$.
Together with~\eqref{WWW}, this implies that 
\[
|\weights{\phi}{x}{K}|\leq 8 n^{i}|\weights{\psi}{x}{K}| / 7k_{r-1}\leq 2 n^{i+1}/r^{i+1}k_r.\qedhere
\]
\end{proof}

Consider the function $\phi_x$ given by Lemma~\ref{addweightvertex}.
Note that for each $y \in N(x)$
\[
\sum_{K\in \KK_r \colon xy\in E(K)} \weights{\phi}{x}{K}= 1 - \tau_{x,y}.
\]
To modify $\phi_x$ into a vertex-gadget, we will add weight $\tau_{x,y}$ to each edge $xy$ using our edge-gadgets. 
This is achieved by the next lemma.

%Thus the vertex-gadget $\xi_x = \phi_x + \sum_{y \in N(x)} \tau_{x,y} \psi_{xy}$, where $\psi_{e}$ is an edge-gadget for an edge~$e$.

\begin{lemma}\label{addweightvertex2} 
Let $r\geq 4$, $0 < \d \leq 1/600r^{3/2}$ and $n\geq 32r^3$. 
Let $G$ be a graph on $n$ vertices with $\d(G)\geq (1-\d)n$.
Let $X:=\{x \in V(G):d_G(x)\geq (1-\d)n+r-1\}$ and suppose $|X|\leq \d (r-1) n$.  
Let $\AA:=\{K\in \KK_r:|V(K)\cap X|\leq r^{1/2}+2\}$. Then for each vertex $x \in V(G)$, there exists a function $\xi_x:\AA\to \mathbb{R}$ so that the following holds.

\begin{enumerate}[label = \rm{(\roman{enumi})}]
\item If $x\in V(G)$ and $e\in E(G)$, then\label{lab1}
\[
\sum_{K\in \AA\colon e\in E(K)}\weights{\xi}{x}{K}=\begin{cases} 1 & \text{if }x\in V(e),\\ 0 & \text{if }x\notin V(e). \end{cases}
\]

\item \label{lab2} For all $K\in \AA$ and $x\in V(G)$, if $i=|V(K)\cap \{x\}|$, then $|\weights{\xi}{x}{K}|\leq 80n^{i+1}/r^{i+1}k_r$.

\end{enumerate}
\end{lemma}

The efficiency of a vertex-gadget $\xi_x$ from Lemma~\ref{addweightvertex2} can be compared to the efficiency of an edge-gadget $\psi_{xy}$ from Lemma~\ref{addweightedge} as follows. 
If a clique $K$ is disjoint from $\{x,y\}$, then (ii) in Lemma~\ref{addweightedge} says that $|\psi_{xy}(K)|\leq 6/k_r$, while (ii) in Lemma~\ref{addweightvertex2} says that $|\xi_x(K)|\leq 80 n/rk_r$; so $\xi_x$ may change the weight of the clique by an extra factor of $n/r$. 
However, $\psi_{xy}$ changes the weight of only one edge by $1$, while $\xi_x$ changes the weight of $|N(x)|\geq (1-\d) n$ edges by $1$. 
As the edge-gadgets can move a $\Omega(1/r^2)$ proportion of the weight, this indicates that the vertex-gadgets can move a $\Omega(1/r)$ proportion of the weight.

\begin{proof} 
For each vertex $x \in V(G)$, let $\phi_x$ be the function from Lemma~\ref{addweightvertex} for which \ref{bad1}--\ref{bad4} hold with the set $\AA$. The function~$\phi_x$ is an approximation to the function $\xi_x$ we require. For each $x\in V(G)$ and $y\in N(x)$, we let
\begin{equation}\label{TAU}
\tau_{x,y}:=1-\sum_{K\in \AA\colon xy\in E(K)}\weights{\phi}{x}{K}.
\end{equation}
As discussed above, this records the adjustments we will need to make to $\phi_x$ in order to obtain $\xi_x$. We will make these adjustments using Lemma~\ref{addweightedge}.

For each $e\in E(G)$, let
\begin{comment} $\HH_{e}$ be the set of sets $A\subset V(G)\setminus V(e)$ for which $G[A\cup V(e)]\in K_{r+2}$ and $|A\cap X|\leq r^{1/2}$.
\end{comment}
\[
\HH_{e} := \{ A\subset V(G)\setminus V(e) : G[A\cup V(e)]\in K_{r+2}\text{ and }|A\cap X|\leq r^{1/2}\}.
\]
By Proposition~\ref{KXsmall}, there are at most $k_{r}/r^2$ sets $A\subset V(G)$ with $G[A]\in \KK_r$ and $|A\cap X|\geq r^{1/2}$.
For each edge $xy\in E(G)$, using Lemma~\ref{cliqest}\ref{est1}, we have
\begin{align}
|\HH_{xy}|&\geq\kappa_{xy}^{(r+2)}-k_r/r^2
\geq k_r-4\d (r+2) k_r-k_r/r^2\geq 3k_r/4. \label{HHxybound}
\end{align}
\begin{comment}
 \nonumber\\
&\geq k_r-2\d nk_{r-1}-2k_r/r^2\geq (1-4\d r-2/r^2)k_{r}\geq 3k_r/4.
|\HH_{xy}|&\geq|\{K\in \KK_r:V(K)\subset N(x)\cap N(y)\}|-k_r/r^2\nonumber\\
&=k_r-|\{K\in \KK_r:V(K)\cap (N^c(x)\cup N^c(y))\neq\emptyset\}|-k_r/r^2\nonumber\\
&\geq k_r-2\d nk_{r-1}-2k_r/r^2\geq (1-4\d r-2/r^2)k_{r}\geq 3k_r/4.\label{HHxybound}
\end{comment}
 For each $e\in E(G)$ and $x\in V(G)$, let 
\begin{equation}\label{Hex}
\HH_{e,x}:=\Big\{A\in \HH_{e}:\sum_{y\in (A\cup V(e))\cap N(x)}|\tau_{x,y}|\leq 12 \Big\}.
\end{equation}

\begin{claim}\label{Hxybound} For all $e\in E(G)$ and $x\in V(G)$, $|\HH_{e,x}|\geq k_r/2$.
\end{claim}
\begin{proof}[Proof of Claim~\ref{Hxybound}]
For each $i\in \{0,1\}$ and each $e\in E(G)$, each $y\in V(G)$ with $|\{y\}\cap V(e)|=i$ is in at most $k_{r+i-1}$ of the sets $A\cup V(e)$ with $A\in \HH_e$.
We therefore have for all $e\in E(G)$ and $x\in V(G)$, that
\begin{align*}
12|\HH_{e}\setminus\HH_{e,x}|& \stackrel{\mathclap{\eqref{Hex}}}{\leq} \sum_{A\in \HH_e\setminus \HH_{e,x}}\sum_{y\in (A\cup V(e))\cap N(x)}|\tau_{x,y}|
\leq \sum_{y\in N(x)\cap V(e)}|\tau_{x,y}|k_{r}+\sum_{y\in N(x)\setminus V(e)}|\tau_{x,y}|k_{r-1}\\
&\stackrel{\mathclap{\ref{bad2}, \ref{bad3}}}{\leq} 2k_{r}/r^{1/2}+nk_{r-1}/r 
\stackrel{\text{P\ref{cliqnos}}}{\leq} 3k_{r}.
\end{align*}
Therefore, $|\HH_e\setminus \HH_{e,x}|\leq k_r/4$. Thus, by (\ref{HHxybound}), $|\HH_{e,x}|\geq 3k_r /4-k_r/4\geq k_r/2$.
\end{proof}

For each $x\in V(G)$, let
\begin{equation}\label{Axdefn}
\AA_x:=\Big\{ K\in \AA:\sum_{y\in V(K)\cap N(x)}|\tau_{x,y}|\leq 12 \Big\}.
\end{equation}
For all $e \in E(G)$, $x\in V(G)$, $A\in \HH_{e,x}$ and cliques $K \in \KK_r[A\cup V(e)]$, we have by (\ref{Hex}) and the definition of $\AA$, $\AA_x$, $\HH_e$ and $\HH_{e,x}$, that $K\in \AA_x$.%
	\COMMENT{
	Note that $A \in \HH_{e}$ implies that $| ( A \cup V(e) ) \cap X | \le r^{1/2}+2$.
	So $|V(K) \cap X| \le r^{1/2}+2$ implying that $K \in \AA$.
	Note that by the definition of~$\HH_{e,x}$, $\sum_{y\in V(K)\cap N(x)} | \tau_{x,y}| \le \sum_{y\in (A \cup V(e)) \cap N(x)} | \tau_{x,y}|\leq 12$.
	So $K \in \AA_x$.
	}
Together with Claim~\ref{Hxybound}, this implies that $\AA_x$ is well-distributed.
Thus, for each $x\in V(G)$ and each $e \in E(G)$, by Lemma~\ref{addweightedge}, there exists a function $\psi^x_e:\AA_x\to \mathbb{R}$ so that the following hold.
\begin{enumerate}[label= (\alph{enumi})]
\item If $e,e'\in E(G)$, then $\sum_{K\in \AA_x\colon e'\in E(K)}\psi_e^{x}(K)=\mathbf{1}_{\{e'=e\}}$. \label{med1}
\item For all $K\in \AA_x$ and $e\in E(G)$, if $i=|V(K)\cap V(e)|$, then $|\psi_e^{x}(K)|\leq 6n^{i}/r^{i}k_r$.\label{med2}
\end{enumerate}
For all $x\in V(G)$ and $e \in E(G)$, extend $\psi_e^x$ by setting $\psi_e^x (K) := 0$ for each $K\in \AA\setminus \AA_x$.
For all $K\in \AA$ and $x\in V(G)$, let
\begin{equation}\label{xidefn}
\weights{\xi}{x}{K}:=\weights{\phi}{x}{K}+\sum_{z\in N(x)} \psi_{xz}^{x}(K) \tau_{x,z}.
\end{equation}

We will now show that the functions $\xi_x$ have the required properties.
Consider any $x \in V(G)$. 
Firstly, for each $y\in N(x)$, by (\ref{xidefn}), \ref{med1} and (\ref{TAU}),
\begin{align*}
\sum_{K\in \AA\colon xy\in E(K)}\weights{\xi}{x}{K}
&=\sum_{K\in \AA\colon xy\in E(K)}\weights{\phi}{x}{K}+\sum_{z\in N(x)}\mathbf{1}_{\{xz=xy\}}\tau_{x,z}
\\
&=\sum_{K\in \AA\colon xy\in E(K)}\weights{\phi}{x}{K}+\tau_{x,y}=1,
\end{align*}
and for each edge $e\in E(G)$ with $x\notin V(e)$, by \ref{bad1} and \ref{med1}, 
\[
\sum_{K\in \AA\colon e\in E(K)}\weights{\xi}{x}{K}=\sum_{K\in \AA\colon e\in E(K)}\weights{\phi}{x}{K}+\sum_{z\in N(x)}\mathbf{1}_{\{xz=e\}}\tau_{x,z}=0+0=0.
\]
Therefore,~\ref{lab1} is satisfied.

It remains to prove~\ref{lab2} for all $x\in V(G)$ and $K\in \AA$, which we do separately for $K\in\AA\setminus\AA_x$ and $K\in \AA_x$.

If $K\in \AA\setminus \AA_x$, then $\psi_e^x (K)=0$ for each $e\in E(G)$.
Therefore, by~(\ref{xidefn}), $\weights{\xi}{x}{K}=\weights{\phi}{x}{K}$. Together with \ref{bad4}, this in turn implies that, if $i=|V(K)\cap \{x\}|$, then $|\weights{\xi}{x}{K}|\leq 2n^{i+1}/r^{i+1}k_r$.

If $K\in \AA_x$, then let $i:=|V(K)\cap \{x\}|$. Note that if $z\in N(x)$ then $|\{x,z\}\cap V(K)|=i+|\{z\}\cap V(K)|$. 
Together with (\ref{xidefn}),~\ref{bad4} and~\ref{med2}, this implies that
\begin{align*}
|\weights{\xi}{x}{K}|
&\leq 2n^{i+1}/r^{i+1}k_r + \sum_{z\in N(x)\cap V(K)}|\tau_{x,z}|(6n^{i+1}/r^{i+1} k_r)+ \sum_{z\in N(x)\setminus V(K)}|\tau_{x,z}|(6n^i/r^ik_r)
\\
&\stackrel{\mathclap{(\ref{Axdefn}),~\ref{bad3}}}{\leq} 2n^{i+1}/r^{i+1}k_r + 12(6n^{i+1}/r^{i+1}k_r)+ 6n^{i+1}/r^{i+1}k_r
=80n^{i+1}/r^{i+1}k_r.\qedhere
\end{align*}
\end{proof}

% ------------------------------ Cliques containing specified edge ---------------------------------

\section{Number of cliques containing a specified edge} \label{sec:kappa}

Recall that, after some initial preprocessing of the graph $G$, we give each $r$-clique weight $1/\kappa$, where $\kappa:=k_{r-2}-2\d nk_{r-3}$. 
This is not far from a fractional $K_r$-decomposition, and we aim to transform it into a fractional $K_r$-decomposition by correcting the weight over each edge using edge- and vertex-gadgets. 
For Theorem~\ref{fracdecomp} we will have $\delta=\Theta(1/r^{3/2})$, and as before we may need to move a $\Omega(\d^2r^2)=\Omega(1/r)$ proportion of the weight around to correct the weight over each edge. 
Our best technique is to use vertex-gadgets which are indeed capable of moving a $\Omega(1/r)$ proportion of the weights over the edges, but only certain adjustments can be made using such gadgets.

The adjustment to be made to the weight over each edge $xy$ is $(\kappa^{(r)}_{xy}/\kappa)-1$. 
In this section, we will break this adjustment down into $\sigma^*(x)+\sigma^*(y)+\pi^*(xy)$ so that on average $\sigma^*(x)=O(1/r)$ and $\pi^*(xy)=O(1/r^2)$. 
Hence we will be able to adjust the weight over each edge $xy$ by $\sigma^*(x)+\sigma^*(y)$ using vertex-gadgets and by $\pi^*(xy)$ using edge-gadgets.

We find such functions in the following lemma (where $(\sigma+\gamma)/\kappa$ and $\pi/\kappa$ correspond to $\sigma^*$ and $\pi^*$), before showing that the error term depending on the edges is $r$-smooth in Lemma~\ref{breakdown2}, so that it can be corrected using the edge-gadgets (via Lemma~\ref{smoothweight}).

\begin{comment}
Let $G$ be a graph. 
Let $\kappa:=k_{r-2}-2\d n k_{r-3}$.
Recall that initially we give each clique $K \in \KK_r$ weight $\kappa^{-1}$.
So each edge~$e$ has weight $\kappa^{(r)}_e /\kappa$.
In this section, our aim is to find functions $\sigma^* : V(G) \to \R$ and $\pi : E(G) \to \R$ such that for each edge $xy \in E(G)$, we have $\kappa^{(r)}_{xy} = \kappa + \sigma^*(x)+\sigma^*(y)+ \pi(xy)$.
We now, for each $x \in V(G)$, apply vertex-gadgets to remove weight $\sigma^*(x)/ \kappa $ from each edge $xy$ with $y \in N(x)$, followed by, for each $e \in E(G)$, apply edge-gadget to remove weight $\pi(e)$ from each edge~$e$.
Notice that the resulting weighting $\omega^*$ on $\KK_r$ is indeed a fractional $K_r$-decomposition (providing that $0 \le\omega^*(K) \le 1$ for all $K \in \mathcal{K}_r$).
We find the functions $\sigma^*$ and $\pi$ in the following lemma.
\end{comment}
In this section, we additionally require the notation that, for sets $A,B\subset V(G)$, 
\[
\bar{e}(A,B):=|\{(x,y):x\in A, y\in B, xy\notin E(G)\}|.
\]

\begin{lemma}\label{breakdown}
Let $r\geq 5$ and $\d:=1/10^4r^{3/2}$. 
Suppose that~$G$ is a graph on~$n$ vertices with $\d(G)\geq (1-\d)n$. Let $X:=\{x \in V(G):d_G(x)\geq (1-\d)n+r-1\}$ and suppose that $|X|\leq \d (r-1) n$.
For each $x\in V(G)$, let $\gamma(x):=(\d n-|N^c(x)|)k_{r-3}$. 
Let $\kappa:=k_{r-2}-2\d n k_{r-3}$.
Let $\pi_1,\pi_2:E(G)\to\R$ be functions defined by
\begin{align}
\pi_1&(xy) := 
 \d n  \sum_{z_1\in N^c(x)} |N^c(z_1)| k_{r-5} + \d n \sum_{z_2\in N^c(y)} |N^c(z_2)| k_{r-5} \nonumber \\
&- \sum_{z_1\in N^c(x)} \sum_{z_2\in N^c(y)}  | N^c(z_1) \cup N^c(z_2) |  k_{r-5}  +(\d n-|N^c(x)|)(\d n-|N^c(y)|)k_{r-4}
  \label{AA9a}
\end{align}
and 
\begin{align}
\pi_2(xy):= \Big( e(N^c(x))(|N^c(y)|-\d n) + e(N^c(y))(|N^c(x)|-\d n) \Big)k_{r-5}.\label{AA17}
\end{align}
%, and for each edge $e\in E(G)$, let $z_e$ be the number of $K_r$ subgraphs in~$G$ containing $e$. 
Then there exist functions $\sigma:V(G)\to \mathbb{R}$ and $\pi:E(G)\to \mathbb{R}$ so that the following hold.
\begin{enumerate}[label= \rm{(\roman{enumi})}]
\item For each $xy\in E(G)$, \label{break1}
\[
\kappa_{xy}^{(r)}=\kappa+\gamma(x)+\gamma(y)+\sigma(x)+\sigma(y)+\pi(xy).
\]
\item For each $x\in V(G)$, $|\sigma(x)|\leq k_{r-2}/10^4r$. \label{break2}

\item \label{break3} 
%If $\pi_1,\pi_2:E(G)\to\R$ are defined for each $xy\in E(G)$ by
%\begin{align}
%\pi_1(x&y):=(\d n-|N^c(y)|)\sum_{z_1\in N^c(x)}|N^c(z_1)|k_{r-5}+(\d n - |N^c(x)|)\sum_{z_2\in N^c(y)}|N^c(z_2)|k_{r-5}\nonumber
%\\
%&+\sum_{z_1\in N^c(x)}\sum_{z_2\in N^c(y)}|N^c(z_1)\cap N^c(z_2)|k_{r-5}+(\d n-|N^c(x)|)(\d n-|N^c(y)|)k_{r-4},\label{AA9}
%\end{align}
%and 
%\begin{align}
%\pi_2(xy):=e(N^c(x))(|N^c(y)|-\d n)k_{r-5}+e(N^c(y))(|N^c(x)|-\d n)k_{r-5},\label{AA17}
%\end{align}
For each $xy\in E(G)$,
\begin{align}
|\pi(xy)|\leq |\pi_1(xy)|&+|\pi_2(xy)|+2|N^c(x)\cap N^c(y)|k_{r-3} \nonumber 
\\&+ 203(\d r)^4k_{r-2} + 3 \bar{e}(N^c(x),N^c(y))k_{r-4}.\nonumber%\label{AA22}
\end{align}
\end{enumerate}
\end{lemma}

\begin{proof} 
By Lemma~\ref{cliqest}\ref{est3}, we have for each $xy\in E(G)$ that
\begin{align*}
\Big|\kappa_{xy}^{(r)}-k_{r-2}-\sum_{i=1}^3(-1)^i \sum_{Z\subset N^c(x)\cup N^c(y)\,:\,|Z|=i} \kappa_Z^{(r-2)}\Big|\leq 11(\d r)^4k_{r-2}.%\label{AA1}
\end{align*}
%The number of cliques in $\KK_{r-2}$ containing some vertex in $N^c(x)\cap N^c(y)$ is at most $|N^c(x)\cap N^c(y)|k_{r-3}$. Therefore, by
Together with Proposition~\ref{cliqnos} this implies that, for each $xy\in E(G)$,%
\COMMENT{To see the first inequality below, note that if $|Z|=3$ and $Z\subset N^c(x)\cap N^c(y)$ then there are 8 ways to have $Z_1\cup Z_2$.
Actually, instead of $7\sum_{j=0}^2\sum_{z\in N^c(x)\cap N^c(y)}\sum_{Z'\subseteq (N^c(x)\cup N^c(y))\setminus \{z\}\,:\,|Z'|=j} \kappa_{\{z\}\cup Z'}^{(r-2)}$ we get $\sum_{j=0}^2\sum_{z\in N^c(x)\cap N^c(y)}\sum_{Z'\subseteq (N^c(x)\cup N^c(y))\setminus \{z\}\,:\,|Z'|=j} (2^{j+1}-1)\kappa_{\{z\}\cup Z'}^{(r-2)}$ which would yield $2|N^c(x)\cap N^c(y)|k_{r-3}$ instead of $8|N^c(x)\cap N^c(y)|k_{r-3}$ in the final estimate.} % end comment
\begin{align}
\Big|\kappa_{xy}^{(r)}-k_{r-2}&-\sum_{i=1}^3 (-1)^i\sum_{j=0}^i\;\sum_{Z_1\subset N^c(x)\,:\,|Z_1|=j}\;\sum_{Z_2\subset N^c(y)\setminus Z_1\,:\,|Z_2|=i-j} \kappa_{Z_1\cup Z_2}^{(r-2)}\Big|\label{AAsum}
\\
&\leq  11(\d r)^4k_{r-2}+
\sum_{j=0}^2\, (2^{j+1}-1)\sum_{z\in N^c(x)\cap N^c(y)}\, \sum_{Z'\subseteq (N^c(x)\cup N^c(y))\setminus \{z\}\,:\,|Z'|=j} \kappa_{\{z\}\cup Z'}^{(r-2)}\nonumber\\
& \le 11(\d r)^4k_{r-2}+|N^c(x)\cap N^c(y)|\Big(k_{r-3}+3\cdot 2\delta n k_{r-4}+7\binom{2\delta n}{2} k_{r-5}\Big)\nonumber \\
& \le 11(\d r)^4k_{r-2}+|N^c(x)\cap N^c(y)|(1+12\delta r+56(\delta r)^2)k_{r-3}\nonumber \\
& \le 11(\d r)^4k_{r-2}+2|N^c(x)\cap N^c(y)|k_{r-3},\nonumber%\label{AA2}
\end{align}
where in the first inequality we are bounding the extra contribution to the sum from those $Z_1 \cup Z_2$ that meet $N^c(x) \cap N^c(y)$.
Thus for each $xy\in E(G)$, we have
\begin{equation}%\label{breakdownequ}
| \kappa_{xy}^{(r)}-k_{r-2}-S_1(x)-S_1(y)-S_2-S_3 | \leq  11(\d r)^4k_{r-2}+2|N^c(x)\cap N^c(y)|k_{r-3},\label{AA2}
\end{equation}
where\COMMENT{A: In the definition of $S_2$, we write $\kappa_{\{z_1,z_2\}}$ instead of $\kappa_{z_1z_2}$ as we do not know whether $z_1z_2$ is an edge.}
\begin{align}\label{S2}
S_2 &= S_2(xy):=\sum_{z_1\in N^c(x)}\sum_{z_2\in N^c(y)\setminus\{z_1\}}\kappa_{\{z_1,z_2\}}^{(r-2)}\\
\label{S3}
S_3 &= S_3(xy) :=-\sum_{j=1}^2\;\sum_{Z_1\subset N^c(x):|Z_1|=j}\;\sum_{Z_2\subset N^c(y)\setminus Z_1:|Z_2|=3-j} \kappa_{Z_1\cup Z_2}^{(r-2)}\\
\intertext{and, for each $z\in V(G)$,}
\nonumber%\label{S1z}
S_1(z) &:=\sum_{i=1}^3 (-1)^i\sum_{Z\subset N^c(z):|Z|=i}\kappa_{Z}^{(r-2)}.
\end{align}
Here $S_1(x)$ and $S_1(y)$ count the contributions to the sum in \eqref{AAsum} from those $Z_1 \cup Z_2$ with one of $Z_1$ or $Z_2$ empty,
and $S_2, S_3$ count the contributions from those $Z_1 \cup Z_2$ with $Z_1, Z_2$ both non-empty and $|Z_1 \cup Z_2| = 2$ or $3$ respectively.
%\begin{align}
%\Bigg|\kappa&_{\{x,y\}}^{(r)}-k_{r-2}-\sum_{i=1}^3 (-1)^i\sum_{Z\subset N^c(x):|Z|=i}\kappa_{Z}^{(r-2)}-\sum_{i=1}^3 (-1)^i\sum_{Z\subset N^c(y):|Z|=i}\kappa_{Z}^{(r-2)}\nonumber
%\\
%&-\sum_{z_1\in N^c(x)}\sum_{z_2\in N^c(y)\setminus\{z_1\}}\kappa_{\{z_1,z_2\}}^{(r-2)}+\sum_{j=1}^2\;\sum_{Z_1\subset N^c%(x):|Z_1|=j}\;\sum_{Z_2\subset N^c(y)\setminus Z_1:|Z_2|=3-j} \kappa_{Z_1\cup Z_2}^{(r-2)}\Bigg|\nonumber
%\\
%&\qquad\qquad \qquad\qquad \qquad\qquad \hfill\leq  12(\d r)^4k_{r-2}+|N^c(x)\cap N^c(y)|k_{r-3}.
%\end{align}
In order to estimate $\kappa_{xy}^{(r)}$ we will now estimate $S_1(x)$, $S_1(y)$, $S_2$, and $S_3$.

We will first estimate $S_1(x)$, for each $x\in V(G)$, for which we let
\begin{align}\label{AA3}
\sigma_1(x)& := S_1(x)-\gamma(x)+\d nk_{r-3} \\
\label{AA3a}
& = \Big( -\sum_{z\in N^c(x)}\kappa_{\{z\}}^{(r-2)}+|N^c(x)|k_{r-3} \Big) + \sum_{i=2}^3 (-1)^i\sum_{Z\subset N^c(x):|Z|=i}\kappa_{Z}^{(r-2)}. 
\end{align}

\begin{claim}\label{sig1}
For each $x\in V(G)$, $|\sigma_1(x)|\leq 8 (\d r)^2k_{r-2}$. 
\end{claim}
\begin{proof}[Proof of Claim~\ref{sig1}]
By Lemma~\ref{cliqest}\ref{est1}, for each $z\in V(G)$, $|\kappa_{\{z\}}^{(r-2)}-k_{r-3}|\leq 2\d r k_{r-3}$. Together with Proposition~\ref{cliqnos}, this implies that
\begin{equation}\label{AA4}
\Big|\sum_{z\in N^c(x)}\kappa_{\{z\}}^{(r-2)}-|N^c(x)|k_{r-3}\Big|\leq \d n \cdot 2\d rk_{r-3}\leq 4(\d r)^2k_{r-2}.
\end{equation}
Moreover, using Proposition~\ref{cliqnos},
\begin{equation}\label{AA5}
\Big|\sum_{i=2}^3 (-1)^i\sum_{Z\subset N^c(x):|Z|=i}\kappa_{Z}^{(r-2)}\Big|
\leq \binom{\d n}{2}k_{r-4}+\binom{\d n}{3}k_{r-5}\leq 4(\d r)^2k_{r-2}.
\end{equation}
The claim follows from (\ref{AA3a}), (\ref{AA4}) and (\ref{AA5}).
\end{proof}

For each $x\in V(G)$, let
\begin{equation}\label{AA10}
\sigma_2(x):=\d n (|N^c(x)|-\d n/2)k_{r-4}-\d n\sum_{z_1\in N^c(x)}|N^c(z_1)|k_{r-5}.
\end{equation}
Note that, by Proposition~\ref{cliqnos}, we have that
\begin{equation}\label{AA11b}
|\sigma_2(x)|\leq (\d n)^2k_{r-4}/2+(\d n)^3k_{r-5}\leq 4\d^2r^2k_{r-2}.
\end{equation}
We will now estimate $|S_2- \sigma_2(x) - \sigma_2(y)|$. If $z_1z_2\in E(G)$, then, by Lemma~\ref{cliqest}\ref{est2},
\begin{equation}\label{AA6b}
\big|\kappa_{\{z_1,z_2\}}^{(r-2)}-k_{r-4}+|N^c(z_1)\cup N^c(z_2)|k_{r-5}\big|\leq 24(\d r)^2k_{r-4}.
\end{equation}
If $z_1z_2\notin E(G)$, then $\kappa_{\{z_1,z_2\}}^{(r-2)}=0$.
Therefore, by (\ref{AA6b}) and Proposition~\ref{cliqnos}, for each $xy\in E(G)$ we have
\COMMENT{middle line of the next calculation 
\begin{align*}
\le  \sum_{z_1\in N^c(x)} \sum_{z_2\in N^c(y) \,:\, z_1z_2\in E(G) } \Bigg| \kappa_{\{z_1,z_2\}}^{(r-2)}
- k_{r-4} + |N^c(z_1)\cup N^c(z_2)|k_{r-5})\Bigg|
\le \sum_{z_1\in N^c(x)} \sum_{z_2\in N^c(y)} 24(\d r)^2k_{r-4}.
\end{align*}
} % end comment
\begin{align*}
\Big|\sum_{z_1\in N^c(x)}\sum_{z_2\in N^c(y)\setminus\{z_1\}}\kappa_{\{z_1,z_2\}}^{(r-2)}
-\sum_{z_1\in N^c(x)}\;&\sum_{z_2\in N^c(y)\,:\, z_1z_2\in E(G)}(k_{r-4}-|N^c(z_1)\cup N^c(z_2)|k_{r-5})\Big|
\\
&\leq 24\d^4r^2n^2k_{r-4}\leq 96( \d r )^4k_{r-2},
\end{align*}
so that, using~\eqref{S2},
\begin{align}
\Big|S_2&
-|N^c(x)||N^c(y)|k_{r-4} +\sum_{z_1\in N^c(x)}\sum_{z_2\in N^c(y)}|N^c(z_1)\cup N^c(z_2)|k_{r-5}\Big|\nonumber
\\
&\leq 96 (\d r)^4k_{r-2}
+\bar{e}(N^c(x),N^c(y)) k_{r-4},\label{AA7}
\end{align}
where we have used the fact that $k_{r-4}\geq |N^c(z_1)\cup N^c(z_2)|k_{r-5}$ by Proposition~\ref{cliqnos}.\COMMENT{$ |N^c(z_1)\cup N^c(z_2)|k_{r-5}  \le 2 \delta n k_{r-5} \le 4 \delta r k_{r-4} \le k_{r-4}$}
%where $\bar{e}(N^c(x),N^c(y))=|\{(z_1,z_2):z_1\in N^c(x),z_2\in N^c(y),z_1z_2\notin E(G)\}|$.
%Note that
%\begin{align}
%\sum_{z_1\in N^c(x)}&\sum_{z_2\in N^c(y)}|N^c(z_1)\cup N^c(z_2)|k_{r-5}\nonumber
%\\
%&=\sum_{z_1\in N^c(x)}\sum_{z_2\in N^c(y)}(|N^c(z_1)|+|N^c(z_2)|-|N^c(z_1)\cap N^c(z_2)|)k_{r-5}\nonumber
%\\
%&=|N^c(y)|\sum_{z_1\in N^c(x)}|N^c(z_1)|k_{r-5}+|N^c(x)|\sum_{z_2\in N^c(y)}|N^c(z_2)|k_{r-5}\nonumber
%\\
%&\quad\quad -\sum_{z_1\in N^c(x)}\sum_{z_2\in N^c(y)}|N^c(z_1)\cap N^c(z_2)|k_{r-5}.\label{AA8}
%\end{align}
%For each $xy\in E(G)$, let
%\begin{align}
%\pi_1(x&y):=(\d n-|N^c(y)|)\sum_{z_1\in N^c(x)}|N^c(z_1)|k_{r-5}+(\d n - |N^c(x)|)\sum_{z_2\in N^c(y)}|N^c(z_2)|k_{r-5}\nonumber
%\\
%&+\sum_{z_1\in N^c(x)}\sum_{z_2\in N^c(y)}|N^c(z_1)\cap N^c(z_2)|k_{r-5}+(\d n-|N^c(x)|)(\d n-|N^c(y)|)k_{r-4},\label{AA9}
%\end{align}
%and 
Note that, by (\ref{AA9a}) and (\ref{AA10}), for each $xy\in E(G)$,
\begin{equation*}%\label{AA10a}
\pi_1(xy)+\sigma_2(x)+\sigma_2(y)=|N^c(x)||N^c(y)|k_{r-4}-\sum_{z_1\in N^c(x)}\sum_{z_2\in N^c(y)}|N^c(z_1)\cup N^c(z_2)|k_{r-5}.
\end{equation*}
Together with \eqref{AA7}, this implies that for each $xy\in E(G)$ we have
\begin{align}
|S_2-\sigma_2(x)-\sigma_2(y)|%&=\Bigg|\Bigg(\sum_{z_1\in N^c(x)}\sum_{z_2\in N^c(y)\setminus\{z_1\}}\kappa_{\{z_1,z_2\}}^{(r-2)}\Bigg)-\sigma_2(x)-\sigma_2(y)\Bigg|
%\nonumber
%\\
&\leq |\pi_1(xy)|+96(\d r)^4k_{r-2}+\bar{e}(N^c(x),N^c(y))k_{r-4}.\label{AA11}
\end{align}

Now, for each $x\in V(G)$, let
\begin{equation}
\sigma_3(x):= -e(N^c(x)) \delta n k_{r-5}.\label{AA16}
\end{equation}
Note that for each $x\in V(G)$, by Proposition~\ref{cliqnos},
\begin{equation}\label{AA19}
|\sigma_3(x)|\leq (\delta n)^3k_{r-5}/2\leq 4\d^3r^3k_{r-2}.
\end{equation}
We will now estimate $|S_3 - \sigma_3(x) - \sigma_3(y)|$. If $G[\{z_1,z_2,z_3\}]\in \KK_3$, then, by Lemma~\ref{cliqest}\ref{est1}, we have $|\kappa_{\{z_1,z_2,z_3\}}^{(r-2)}-k_{r-5}|\leq 6\d rk_{r-5}$. %\nonumber%\label{AA12}
Therefore, 
\begin{align}
\Big|\sum_{z_1\in N^c(x)} & \sum_{\{z_2,z_3\}\subset N^c(y)\setminus\{z_1\}}\kappa^{(r-2)}_{\{z_1,z_2,z_3\}} -e(N^c(y))|N^c(x)|k_{r-5}\Big| \nonumber \\
& = \Big|\sum_{z_1 \in N^c(x)} \Big( \sum_{\{z_2,z_3\}\subset N^c(y)\setminus\{z_1\}}\kappa^{(r-2)}_{\{z_1,z_2,z_3\}}
-\sum_{\{z_2,z_3\}\subset N^c(y):z_2z_3\in E(G)}k_{r-5} \Big) \Big| \nonumber%\label{AA13}
\\
&\leq \bar{e}(N^c(x),N^c(y)) (\d n)\cdot k_{r-5}+(\d n)^3\cdot 6\d r k_{r-5} \nonumber \\
%&\leq 2 \d r \bar{e}(N^c(x),N^c(y))k_{r-4}+48(\d r)^4k_{r-2}
&\leq \bar{e}(N^c(x),N^c(y))k_{r-4} +48(\d r)^4k_{r-2},
\label{AA14}
\end{align}
where the last inequality is due to Proposition~\ref{cliqnos} and the fact that $\delta r \le 1/2$.
%
%
%Together with Proposition~\ref{cliqnos}, this implies that
%\begin{align}
%\Bigg|\sum_{z_1\in N^c(x)}\sum_{\{z_2,z_3\}\subset N^c(y)\setminus\{z_1\}}\kappa^{(r-2)}_{\{z_1,z_2,z_3\}}
%&-e(N^c(y))|N^c(x)|k_{r-5}\Bigg|  \label{AA14}
%\\
%&\leq \bar{e}(N^c(x),N^c(y))(\d n)k_{r-5}+48(\d r)^4k_{r-2}.\nonumber
%\end{align}
Similarly,
\begin{align}
\Big|\sum_{z_1\in N^c(y)}\sum_{\{z_2,z_3\}\subset N^c(x)\setminus\{z_1\}}\kappa^{(r-2)}_{\{z_1,z_2,z_3\}}
&-e(N^c(x))|N^c(y)|k_{r-5}\Big| \nonumber \\
& \le \bar{e}(N^c(x),N^c(y))k_{r-4} +48(\d r)^4k_{r-2}. \label{AA15}
\end{align}
%and for each edge $xy\in E(G)$, let
%\begin{align}
%\pi_2(xy):=e(N^c(x))(|N^c(y)|-\d n)k_{r-5}+e(N^c(y))(|N^c(x)|-\d n)k_{r-5}.\label{AA17}
%\end{align}
Note that, by (\ref{AA17}) and (\ref{AA16}), for each $xy\in E(G)$,
\begin{equation*}%\label{AA17b}
\pi_2(xy)-\sigma_3(x)-\sigma_3(y)=e(N^c(x))|N^c(y)|k_{r-5}+e(N^c(y))|N^c(x)|k_{r-5}.
\end{equation*}
Together with~\eqref{S3}, (\ref{AA14}) and (\ref{AA15}), this implies that\COMMENT{From below: $=\Bigg|-\Bigg(\sum_{j=1}^2\;\sum_{Z_1\subset N^c(x):|Z_1|=j}\;\sum_{Z_2\subset N^c(y)\setminus Z_1:|Z_2|=3-j} \kappa_{Z_1\cup Z_2}^{(r-2)}\Bigg)-\sigma_3(x)-\sigma_3(y)\Bigg|$}
\begin{align}
|S_3-\sigma_3(x)-&\sigma_3(y)|
\leq |\pi_2(xy)|+ 2 \bar{e}(N^c(x),N^c(y))k_{r-4}+96(\d r)^4k_{r-2}.\label{AA18}
\end{align}

For each $x\in V(G)$, let
\begin{equation}\label{AA20}
\sigma(x):=\sigma_1(x)+\sigma_2(x)+\sigma_3(x),
\end{equation}
and for each edge $xy\in E(G)$, let
\begin{equation}\label{AA21}
\pi(xy):=\kappa^{(r)}_{xy}-\kappa-\gamma(x) -\gamma (y)-\sigma(x)-\sigma(y).
\end{equation}
Then \ref{break1} holds.
Note that, for each $x\in V(G)$, by Claim~\ref{sig1}, (\ref{AA11b}), (\ref{AA19}) and (\ref{AA20})%
\COMMENT{Here we use the fact that $\delta \leq 1/400r^{3/2}$.}
\[
|\sigma(x)|\leq 8\d^2r^2k_{r-2}+4\d^2r^2k_{r-2}+4\d^3r^3k_{r-2} \leq 13 \delta^2 r^2 k_{r-2} \leq k_{r-2}/10^4r,
\]
and thus \ref{break2} holds.

Note that $\pi(xy)=\kappa_{xy}^{(r)}-k_{r-2}-S_1(x)-S_1(y)-\sum_{i=2}^3(\sigma_i(x)+\sigma_i(y))$ by (\ref{AA3}), \eqref{AA20} and \eqref{AA21}. Together with (\ref{AA2}), (\ref{AA11}) and (\ref{AA18}) this shows that for each $xy\in E(G)$ we have that\COMMENT{full calculation
\begin{align*}
	|\pi(xy)|
\end{align*}
\begin{align*}
& \le
	\Big| \kappa_{xy}^{(r)}-k_{r-2}-S_1(x)-S_1(y)-S_2-S_3 \Big| \\
	& \qquad + |S_2- \sigma_2(x) - \sigma_2(y)| + |S_3-\sigma_3(x)-\sigma_3(y)|\\
	& \le  \big( 11(\d r)^4k_{r-2}+2|N^c(x)\cap N^c(y)|k_{r-3} \big) \\
	& \qquad +\big( |\pi_1(xy)|+96\d^4r^4k_{r-2}+\bar{e}(N^c(x),N^c(y))k_{r-4} \big) 
	\\ & \qquad + \big( |\pi_2(xy)|+ 2\bar{e}(N^c(x),N^c(y))k_{r-4}+96(\d r)^4k_{r-2}\big)\\
	&\le |\pi_1(xy)|+|\pi_2(xy)|+2|N^c(x)\cap N^c(y)|k_{r-3} + 203(\d r)^4k_{r-2} + 3 \bar{e}(N^c(x),N^c(y))k_{r-4}
\end{align*}
}
\begin{align*}
	|\pi(xy)|\hspace{14.25cm}
\end{align*}
\vspace{-0.75cm}
\begin{align*}
\hspace{0.275cm}& \le
	\Big| \kappa_{xy}^{(r)} -k_{r-2}-S_1(x)-S_1(y)-S_2-S_3 \Big| 
+ \Big|S_2- \sigma_2(x) - \sigma_2(y) \Big| + \Big| S_3-\sigma_3(x)-\sigma_3(y) \Big| \\
	&\le |\pi_1(xy)|+|\pi_2(xy)|+2|N^c(x)\cap N^c(y)|k_{r-3} + 203(\d r)^4k_{r-2} + 3 \bar{e}(N^c(x),N^c(y))k_{r-4}
\end{align*}
and thus \ref{break3} holds.
\end{proof}

%\begin{align*}	 
%	\Big| \kappa_{xy}^{(r)}& -k_{r-2}-S_1(x)-S_1(y)-S_2-S_3 \Big| 
%+ \Big|S_2- \sigma_2(x) - \sigma_2(y) \Big| + \Big| S_3-\sigma_3(x)-\sigma_3(y) \Big| \\
%	&\le |\pi_1(xy)|+|\pi_2(xy)|+2|N^c(x)\cap N^c(y)|k_{r-3} + 203(\d r)^4k_{r-2} + 3 \bar{e}(N^c(x),N^c(y))k_{r-4}
%\end{align*}

Given a function $\pi:E(G)\to\R$ with the properties in Lemma~\ref{breakdown}, we wish to use Lemma~\ref{smoothweight} to add the weight $\pi(e)/\kappa$ to each edge $e$. 
We must therefore check that $\pi/\kappa$ is $r$-smooth.

%Next we show that, letting $\kappa:=k_{r-2}-2\delta nk_{r-3}$, the function $\pi/\kappa: E(G) \to \R$ is smooth, so that we can apply Lemma~\ref{smoothweight}.

\begin{lemma}\label{breakdown2}
Let $r\geq 25$, $\d:=1/10^4r^{3/2}$ and $n\geq 10^{4}r^3$.\COMMENT{A: Actually, we can have $\d:=1/(600r^{3/2})$ but we would need $r \ge 5000$. 
The main constraint comes from Claim~\ref{BB1}.} 
Suppose that~$G$ is a graph on~$n$ vertices with $\d(G)\geq (1-\d)n$. 
Let $X:=\{x \in V(G):d_G(x)\geq (1-\d)n+r-1\}$ and suppose that $|X|\leq \d (r-1) n$. 
Let $\kappa:=k_{r-2}-2\delta nk_{r-3}$ and let $\pi_1,\pi_2:E(G)\to \R$ be the functions defined in the statement of Lemma~\ref{breakdown}.
Suppose that $\pi:E(G)\to \mathbb{R}$ satisfies
\begin{align}
|\pi(xy)|\leq |\pi_1(xy)|&+|\pi_2(xy)|+2|N^c(x)\cap N^c(y)|k_{r-3}\nonumber
\\
&+203(\d r)^4k_{r-2}+3\bar{e}(N^c(x),N^c(y))k_{r-4}.\label{AA22ag}
\end{align}
Then the function $\pi/\kappa$ is $r$-smooth. %\label{break3}
\end{lemma}

\begin{proof} We will show that $\pi/\kappa$ is $r$-smooth using a sequence of claims. Note first, using Proposition~\ref{cliqnos}, that $\kappa\geq k_{r-2}-4\d rk_{r-2}\geq 9k_{r-2}/10$.

\begin{claim}\label{BB1} For each $xy\in E(G)$, $|\pi(xy)|\leq \kappa/10^4$.
That is, $\pi/\kappa$ satisfies \ref{good1} in the definition of $\pi/\kappa$ being $r$-smooth.
\COMMENT{Here we need $4 \d r \le 4/(5\cdot 10^4)$.}
\end{claim}
\begin{proof}[Proof of Claim~\ref{BB1}] Note that $|N^c(z_1)\cup N^c(z_2)|=|N^c(z_1)|+|N^c(z_1)|-|N^c(z_1)\cap N^c(z_2)|$ for each $z_1,z_2\in V(G)$. Therefore, for each $xy\in E(G)$ we have by~\eqref{AA9a}
that  
\begin{align}
\pi_1(x&y)=(\d n-|N^c(y)|)\sum_{z_1\in N^c(x)}|N^c(z_1)|k_{r-5}+(\d n - |N^c(x)|)\sum_{z_2\in N^c(y)}|N^c(z_2)|k_{r-5}\nonumber
\\
&+\sum_{z_1\in N^c(x)}\sum_{z_2\in N^c(y)}|N^c(z_1)\cap N^c(z_2)|k_{r-5}+(\d n-|N^c(x)|)(\d n-|N^c(y)|)k_{r-4}. \label{AA9}
\end{align}
So $|\pi_1(xy)|\leq 3(\d n)^3k_{r-5}+(\d n)^2k_{r-4}$.
By (\ref{AA17}), we have $|\pi_2(xy)|\leq (\d n)^3k_{r-5}$. Therefore, by (\ref{AA22ag}), Proposition~\ref{cliqnos} and the fact that $r^{1/2}\geq 5$, we have
\begin{align}
|\pi(xy)|&\leq 4(\d n)^3k_{r-5}+203(\d r)^4k_{r-2}+4(\d n)^2k_{r-4}+ 2|N^c(x)\cap N^c(y)| k_{r-3}\nonumber
\\
& \leq ( 32(\d r)^3 + 203(\d r)^4 + 16(\d r)^2 )  k_{r-2} +2|N^c(x)\cap N^c(y)|k_{r-3} \nonumber \\
&\leq 20(\d r)^2k_{r-2}+2|N^c(x)\cap N^c(y)|k_{r-3}\label{BB0}
\\
%& \leq (32(\d r)^3+144(\d r)^4+16(\d r)^2+2\d r)k_{r-2}\nonumber
\nonumber
& \leq k_{r-2}/10^5+4\d rk_{r-2}=(1/10^5+4/10^4r^{1/2})k_{r-2}
\\
& \leq 9k_{r-2}/10^5\leq \kappa/10^4.\qedhere
\end{align}
\end{proof}

%Thus, by Claim~\ref{BB1}, $\pi$ satisfies \ref{good1} in the definition of $\pi$ being $(k_{r-2}/10^4r^2,r)$-smooth.

\begin{claim}\label{BB1b} For each vertex $x\in V(G)$, $\sum_{y\in N(x)}|\pi(xy)|\leq \kappa n/10^4r$.\COMMENT{Here we need $4 (\d r)^2 \le 8/ 10^5r$.}
That is, $\pi/\kappa$ satisfies \ref{good2} in the definition of $\pi/\kappa$ being $r$-smooth.
\end{claim}
\begin{proof}[Proof of Claim~\ref{BB1b}] By (\ref{BB0}) and Proposition~\ref{cliqnos}, we have, for each $x\in V(G)$, that
\begin{align*}
\sum_{y\in N(x)}|\pi(xy)|&\leq 20(\d r)^2 n k_{r-2}+2\sum_{y\in N(x)}|N^c(x)\cap N^c(y)|k_{r-3} \\
&\leq nk_{r-2}/10^5r+2\sum_{z\in N^c(x)}|N^c(z)|k_{r-3} 
 \leq nk_{r-2}/10^5r+2\d^2n^2k_{r-3}\\
&\leq nk_{r-2}/10^5r+4\d^2 r nk_{r-2}
\leq 9nk_{r-2}/10^5r\leq \kappa n/10^4r.\qedhere
\end{align*}
\end{proof}

%Thus, by Claim~\ref{BB1b}, $\pi$ satisfies \ref{good2} in the definition of $\pi$ being $(k_{r-2}/10^4r^2,r)$-smooth.

\begin{claim}\label{BB2} We have $\sum_{x\in V(G)}\sum_{y\in N(x)}|\pi_1(xy)|\leq n^2k_{r-2}/10^5r^2$.
\end{claim}
\begin{proof}[Proof of Claim~\ref{BB2}] Note that, as $10^4 \delta^2 n \geq 1$,
\begin{equation}\label{CC0}
\sum_{x\in V(G)}(\d n-|N^c(x)|)\leq \sum_{x\in X}\d n+\sum_{x\notin X}r\leq \d^2 rn^2+rn\leq 10^5\d^2r n^2.
\end{equation}
%Therefore,
%\begin{equation}\label{CC1}
%\sum_{x\in V(G)}\sum_{y\in N(x)}||N^c(y)|-\d n|\leq 2\d^2r n^3.
%\end{equation}
Note also that
\begin{equation}\label{CC2}
\sum_{y\in V(G)}\sum_{z_1\in N^c(y)}|N^c(z_1)|\leq \d n\sum_{z_1\in V(G)}|N^c(z_1)|\leq \d^2 n^3.
\end{equation}
Therefore, by (\ref{CC0}) and (\ref{CC2}),%, and Proposition~\ref{cliqnos},
\begin{align}
\sum_{x\in V(G)}\sum_{y\in N(x)}\Big((\d n-|N^c(x)|)\sum_{z_1\in N^c(y)}|N^c(z_1)|\Big)&\leq 10^5\d^4rn^5. \label{CC3}
\end{align}
Note also that%
\begin{align}
\sum_{x\in V(G)}
\sum_{y\in N(x)}
\sum_{z_1\in N^c(x)}
\sum_{z_2\in N^c(y)}
|N^c(z_1)\cap N^c(z_2)|
& \leq
%\sum_{z_1 \in V(G)}
%\sum_{z_2 \in V(G)}
%\sum_{x \in N^c(z_1)}
%\sum_{y \in N^c(z_2) \cap N^c(y)}
%|N^c(z_1)\cap N^c(z_2)|
%\nonumber \\ &
%\leq
(\delta n)^2
\sum_{z_1 \in V(G)}
\sum_{z_2 \in V(G)}
|N^c(z_1)\cap N^c(z_2)|
\nonumber \\ &
\leq 
(\delta n)^2
\sum_{z \in V(G)}
|N^c(z)|^2
\leq
\delta^4 n^5 \label{CC5}
\end{align}
Furthermore, by (\ref{CC0}),
\begin{align}
\sum_{x\in V(G)}\sum_{y\in N(x)}(\d n-|N^c(x)|)(\d n-|N^c(y)|)&\leq \Big(\sum_{x\in V(G)}(\d n-|N^c(x)|)\Big)^2 
\leq (10^5\d^2 r n^2)^2.\label{CC4}
\end{align}
Therefore, by (\ref{AA9}),  (\ref{CC3}), (\ref{CC5}), (\ref{CC4}) and Proposition~\ref{cliqnos}
\begin{align*}
\sum_{x\in V(G)}\sum_{y\in N(x)}|\pi_1(xy)|& \leq 2 \cdot 10^5\d^4rn^5 k_{r-5} + \d^4n^5 k_{r-5} + (10^5\delta^2 r n^2 )^2 k_{r-4} 
\\
& \leq \big( 16 \cdot 10^5(\d r)^4 + 8 \d^4r^3  + 4\cdot 10^{10}(\delta r )^4    \big) n^2k_{r-2} 
\\
&\le 10^{11} (\d r)^4 n^2k_{r-2}= n^2k_{r-2}/10^5r^2. \qedhere
\end{align*}
\end{proof}

\begin{claim}\label{BB3} We have $\sum_{x\in V(G)}\sum_{y\in N(x)}|\pi_2(xy)|\leq n^2k_{r-2}/10^5r^2$.
\end{claim}
\begin{proof}[Proof of Claim~\ref{BB3}] Note that, from (\ref{AA17}), for each $x\in V(G)$ and $y\in N(x)$,
\begin{align*}
|\pi_2(xy)|\leq (\d n -|N^c(y)|)\d^2n^2k_{r-5}+(\d n -|N^c(x)|)\d^2n^2k_{r-5}.
\end{align*}
Together with (\ref{CC0}) and Proposition~\ref{cliqnos}, this implies that
\begin{align*}
\sum_{x\in V(G)}  \sum_{y\in N(x)}|\pi_2(xy)| & \leq 2 n\sum_{x\in V(G)}(\d n -|N^c(x)|)\d^2n^2k_{r-5}
\\
&\leq 2n\cdot 10^5\d^2 rn^2\cdot \d^2n^2k_{r-5}
\leq 10^7 (\d r)^4n^2k_{r-2}\leq n^2k_{r-2}/10^5r^2. \qedhere
\end{align*}
\end{proof}

\begin{claim}\label{BB4} We have $\sum_{x\in V(G)}\sum_{y\in N(x)}|N^c(x)\cap N^c(y)|k_{r-3}\leq  n^2k_{r-2}/10^5r^2$. 
\end{claim}
\begin{proof}[Proof of Claim~\ref{BB4}]
 We have that, using Proposition~\ref{cliqnos},
\begin{align*}
\sum_{x\in V(G)}\sum_{y\in N(x)}|N^c(x)\cap N^c(y)|k_{r-3}&\leq \sum_{x\in V(G)}\sum_{z\in N^c(x)}|N^c(z)|k_{r-3}\leq \d^2n^3k_{r-3}
\\
&\leq 2\d^2 rn^2 k_{r-2}\leq n^2k_{r-2}/10^5r^2.\qedhere
\end{align*}

\end{proof}

\begin{claim}\label{BB5} We have $\sum_{x\in V(G)}\sum_{y\in N(x)}\bar{e}(N^c(x),N^c(y))k_{r-4}\leq n^2k_{r-2}/10^5r^2$. 
\end{claim}
\begin{proof}[Proof of Claim~\ref{BB5}]
Note that
\begin{align*}
 \sum_{x\in V(G)}\sum_{y\in N(x)}\bar{e}(N^c(x),N^c(y)) 
 \leq |\{(x, z_1, z_2, y) \in V(G)^4 : xz_1, z_1z_2, z_2y \notin E(G)\}| 
 \leq n(\delta n)^3,
\end{align*}
so by Proposition~\ref{cliqnos},
\[
\sum_{x\in V(G)}\sum_{y\in N(x)}\bar{e}(N^c(x),N^c(y))k_{r-4}\leq \d^3 n^4k_{r-4}\leq 4\d^3 r^2 n^2k_{r-2}\leq n^2k_{r-2}/10^5r^2.\qedhere
\]
\end{proof}

Now \eqref{AA22ag} and Claims~\ref{BB2}--\ref{BB5} together imply that
\begin{align*}
2\sum_{e\in E(G)}|\pi(e)| & =  \sum_{x\in V(G)}\sum_{y\in N(x)}|\pi(xy)|
 \leq 203(\d r)^4n^2k_{r-2} + 7n^2k_{r-2}/10^5r^2 \leq 2n^2\kappa/10^4r^2.
\end{align*}
 Thus $\pi/\kappa$ satisfies \ref{good3} in the definition of $\pi/\kappa$ being $r$-smooth. This completes the proof that $\pi/\kappa$ is $r$-smooth.
\end{proof}

%----------------------------- final proof ---------------------------------

\section{Proof of Theorem~\ref{fracdecomp}}\label{secproof}

We now combine our results and techniques to prove Theorem~\ref{fracdecomp}. 
After some initial preprocessing, we give each clique a uniform weighting before using Lemma~\ref{breakdown} to break down the adjustments that need to be made to the weight over each edge. 
We carry out the (potentially) larger adjustments using our vertex-gadgets from Lemma~\ref{addweightvertex2}, while the finer adjustments are shown to be $r$-smooth by Lemma~\ref{breakdown2} and can thus be made using Lemma~\ref{smoothweight}; making these corrections gives a fractional $K_r$-decomposition of the graph.

\begin{proof}[Proof of Theorem~\ref{fracdecomp}] 
First note that, for $r \leq 24$, $1/10^4r^{3/2} \leq 1/64r^3$ (with room to spare), so the result follows from Theorem~\ref{hypergraphs} with $k=2$.
So we may assume that $r \geq 25$.

Let $\d:=1/10^4r^{3/2}$ and $X:=\{x\in V(G):d(x)\geq (1-\d)n+r-1\}$. 
As in the proof of Theorem~\ref{r2lemma}, we may assume that $G[X]$ is $K_r$-free and that, similarly, $|X|\leq \d (r-1) n$.

%For each edge $e\in E(G)$, let $z_e$ be the number of cliques $K\in \KK_r$ with $e\in E(K)$. 
Let $\kappa:=k_{r-2}-2\d nk_{r-3}$, and, for each vertex $x\in V(G)$, let 
\begin{equation*}
\gamma(x):=(\d n-|N^c(x)|)k_{r-3}.
\end{equation*}
By Lemmas~\ref{breakdown} and~\ref{breakdown2}, there are functions $\sigma:V(G)\to \mathbb{R}$ and $\pi:E(G)\to \mathbb{R}$, so that the following hold.
\begin{enumerate}[label= \rm{(\roman{enumi})}]
\item For each edge $xy\in E(G)$, $\kappa_{xy}^{(r)}= \kappa +\gamma(x)+\gamma(y)+\sigma(x)+\sigma(y)+\pi(xy)$.\label{ev1}

\item For each vertex $x\in V(G)$, $|\sigma(x)|\leq k_{r-2}/10^4r$.\label{ev2}

\item The function $\pi/\kappa$ is $r$-smooth.\label{ev3}
\end{enumerate}
By Lemma~\ref{smoothweight}, there exists a weighting $\omega': \KK_r \to \R$ so that the following hold.
\begin{enumerate}[label= \rm{(\roman{enumi})}]\addtocounter{enumi}{3}
\item For each $e\in E(G)$, $\sum_{K\in \KK_r:e\in E(K)} \omega'(K)=\pi(e)/\kappa$. \label{ev6} 
\item For each $K\in \KK_r$, $|\omega' (K) |\leq 1/2\kappa$.\label{ev7}
\end{enumerate}
Let $\AA:=\{K\in \KK_r:|V(K)\cap X|\leq r^{1/2}+2\}$.
By Lemma~\ref{addweightvertex2}, for each $x \in V(G)$, there is a function $\xi_x :\AA \to \R$, so that
\begin{enumerate}[label = \rm{(\roman{enumi})}]\addtocounter{enumi}{5}
\item If $x\in V(G)$ and $e\in E(G)$, then \label{ev4} $\sum_{K\in \AA\colon e\in E(K)}\weights{\xi}{x}{K}=\mathbf{1}_{\{x\in V(e)\}}$.

\item For each $K\in \AA$, and $x\in V(G)$, if $i=|V(K)\cap \{x\}|$, then \label{ev5} $
|\weights{\xi}{x}{K}|\leq 80n^{i+1}/r^{i+1}k_r$.
\end{enumerate}
Extend each $\xi_x$ by letting $\weights{\xi}{x}{K}:=0$ for each $K\in \KK_r\setminus \AA$.
Define a function $\omega: \KK_r \to \R$ by
\begin{equation}\label{FF2}
\omega (K) :=\frac{1}{\kappa}\Big(1- \kappa\cdot\omega'(K) -\sum_{x\in V(G)}(\gamma(x)+\sigma(x))\weights{\xi}{x}{K}\Big).
\end{equation}
We now check that $\omega$ gives a fractional $K_r$-decomposition of~$G$.

Firstly, for each edge $xy\in E(G)$, by~\eqref{FF2}, the definition of $\kappa_{xy}^{(r)}$,~\ref{ev6} and~\ref{ev4}, and then by \ref{ev1}, we have
\[
\sum_{K\in \KK_r:xy\in E(K)}\omega (K) =\frac{1}{\kappa }\Big(\kappa_{xy}^{(r)}-\pi(xy)-\sum_{v\in V(G)}(\gamma(v)+\sigma(v))\mathbf{1}_{\{v\in\{x,y\}\}}\Big)=1.
\]

Secondly note that, for each $x\in V(G)$, $|\gamma(x)|\leq \d n k_{r-3}$, and thus, by \ref{ev2} and Proposition~\ref{cliqnos},\COMMENT{use the fact that $\d = 1/10^4r^{3/2}$ and $ r \ge 5$ in the last inequality}
\begin{equation}\label{R1111}
|\gamma(x)+\sigma (x)|\leq \d n k_{r-3}+k_{r-2}/10^4r\leq (8 \d r^3 + 4  r /10^4 )k_{r}/n^2 \leq 9 r^{3/2} k_{r}  /10^4 n^2.
\end{equation}
Furthermore, if $x\in V(G)\setminus X$, then $|\gamma(x)|\leq rk_{r-3}$, and thus
 by~\ref{ev2}, Proposition~\ref{cliqnos} and the fact that $n \ge 10^4 r^3$,
\begin{equation}\label{R1112}
|\gamma(x)+\sigma (x)|\leq r k_{r-3}+k_{r-2}/10^4r \leq ( 8 r^4 / n + 4 r / 10^4)k_{r}/n^2 \leq  12 r k_{r}/10^4 n^2.
\end{equation}
Therefore, if $K\in \AA$, then, by the definition of $\AA$,~\eqref{R1111}, \eqref{R1112} and the fact that $r\ge 25$,%
   \COMMENT{In the final inequality we use that $(r^{1/2}+2)\cdot 9 k_{r} r^{3/2} /10^4n^2  + r \cdot 12 r k_{r}/10^4 n^2
\le \frac{7r^{1/2}}{5}\cdot 9 k_{r} r^{3/2} /10^4n^2+r \cdot 12 r k_{r}/10^4 n^2\le 25r^2k_r/10^4n^2$ since $r\ge 25$.}
\begin{align}
\sum_{x\in V(K)}|\gamma(x)+\sigma(x)| &\leq \sum_{x\in V(K)\cap X}|\gamma(x)+\sigma(x)|+\sum_{x\in V(K)\setminus X}|\gamma(x)+\sigma(x)|\nonumber
\\
&\leq (r^{1/2}+2)\cdot 9 k_{r} r^{3/2} /10^4n^2  + r \cdot 12 r k_{r}/10^4 n^2\leq  3 r^{2} k_{r}/10^3 n^2.\label{useful1}
\end{align}
Furthermore,~\eqref{R1111},~\eqref{R1112}, and the fact that $|X|\leq \d (r-1)n$ together imply that\COMMENT{use the fact that $\d = 1/10^4r^{3/2}$ in the last inequality}
\begin{align}
\sum_{x\in V(G)}|\gamma(x)+\sigma (x)|& \leq \d (r-1) n \cdot 9 r^{3/2} k_{r} /10^4 n^2 + n \cdot 12 r k_{r}/10^4 n ^2
\leq  2r k_{r}/10^3 n
.\label{useful2}
\end{align}
So for each clique $K\in \AA$, we have
\begin{align}
\Big|\sum_{x\in V(G)}(\gamma(x)+\sigma(x))  \weights{\xi}{x}{K}\Big|  & \stackrel{\mathclap{\ref{ev5}}}{\leq} \sum_{x\in V(K)}|\gamma(x)+\sigma(x)| \frac{80 n^2}{r^2k_r} 
 +\!\!\!\!\sum_{x\in V(G)\setminus V(K)}\!\!|\gamma(x)+\sigma(x)|\frac{80 n^2}{r^2k_r} 
 \nonumber
\\
&\stackrel{\mathclap{(\ref{useful1}),(\ref{useful2})}}{\leq} \quad \frac{3r^{2} k_{r}}{10^3 n^2} \cdot \frac{80n^2}{r^2k_r}+\frac{2r k_{r}}{10^3 n}\cdot\frac{80n}{rk_r}\leq 1/2. \label{FF3}
\end{align}
If $K\in \KK_r\setminus \AA$, then as $\weights{\xi}{x}{K}=0$ for each $x\in V(G)$, we have $|\sum_{x\in V(G)}(\gamma(x)+\sigma(x))\weights{\xi}{x}{K})|=0$. Therefore, by (\ref{FF2}), \ref{ev7}, and (\ref{FF3}), for each $K\in \KK_{r}$, $\omega(K)\geq (1-1/2-1/2)/\kappa\geq 0$, as required.
\end{proof}

\end{document}